\documentclass[12pt]{amsart}
\usepackage{amssymb}
\usepackage{graphicx}
\usepackage{hyperref}
\usepackage[all]{xy}
\usepackage{a4wide}
\usepackage{color}
\fontsize{12}{18}\selectfont

\renewcommand{\theenumi}{\alph{enumi}}

\newtheorem{lemma}{Lemma}[section]
\newtheorem{theorem}[lemma]{Theorem}

\newtheorem{corollary}[lemma]{Corollary}

\newtheorem{proposition}[lemma]{Proposition}
\theoremstyle{remark}

\newtheorem{remark}[lemma]{Remark}

\newtheorem{example}[lemma]{Example}

\newtheorem{conjecture}[lemma]{Conjecture}

\theoremstyle{definition}

\newtheorem{definition}[lemma]{Definition}
{}
\newcounter{Lcounter}
\renewcommand{\theLcounter}{\theequation\alph{Lcounter}}

\newcounter{condcounter}

\newcounter{NumberEquation}

{\end{eqnarray}\setcounter{equation}{\value{NumberEquation}}
\renewcommand{\theequation}{\arabic{equation}}}
\def\wre{\mathop{\rm wr}\nolimits}
\renewcommand{\wr}{\wre}
\renewcommand{\phi}{\varphi}

\newcommand{\Ind}{{\rm Ind}}

\def\st{\mathop{\mid\;}}

\newcommand{\Ram}{{\rm Ram}}

\newcommand{\gal}{\textnormal{Gal}}
\newcommand{\Gal}{\textnormal{Gal}}

\newcommand{\eps}{\varepsilon}

\newcommand{\hefresh}{\smallsetminus}
\newcommand{\fsEP}[6]{\xymatrix{& #1\ar[d]^{#2}\ar@{.>}[ld]_{#6}\\ #3\ar[r]^{#4} & #5}}
\newcommand{\sEP}[4]{\xymatrix{& #1\ar[d]\ar@{.>}[dl]_{#4}\\#2\ar[r] & #3}}

\def\dotunion{
\def\dotunionD{\bigcup\kern-9pt\cdot\kern5pt}
\def\dotunionT{\bigcup\kern-7.5pt\cdot\kern3.5pt}
\mathop{\mathchoice{\dotunionD}{\dotunionT}{}{}}}

\def\normal{\triangleleft}

\def\and{{\rm and}}

\def\gcd{{\rm gcd}}

\def\Hom{{\rm Hom}}

\def\id{{\rm id}}

\def\rank{\mathop{\rm rank}\nolimits}

\def\res{{\rm res}}

\newcommand{\calE}{{\mathcal E}}
\newcommand{\calF}{{\mathcal F}}

       \newcommand{\Egal}{{\tilde E}}

       \newcommand{\Kgal}{{\tilde K}}
       \newcommand{\Lgal}{{\tilde L}}

\newcommand{\gag}[1]{\bar{#1}}
\def\agag{{\bar a}}     
\def\bgag{{\bar b}}

     \def\Egag{{\gag{E}}}
     \def\Fgag{{\gag{F}}}

     \def\Ngag{{\gag{N}}}

     \def\Rgag{{\gag{R}}}
     \def\Sgag{{\gag{S}}}
\def\tgag{{\bar t}}     
     
\def\vgag{{\bar v}}     
     
\def\xgag{{\bar x}}     
\def\ygag{{\bar y}}

\def\betagag{{\bar \beta}}
\def\gammagag{{\bar \gamma}}

\def\mugag{{\bar \mu}}

     \def\Ehat{{\hat E}}
     \def\Fhat{{\hat F}}
\def\ghat{{\hat g}}     \def\Ghat{{\hat G}}

\def\alphahat{{\hat \alpha}}

\def\phihat{{\hat \varphi}}

\def\psihat{{\hat \psi}}

\def\muhat{{\hat \mu}}
\def\thetahat{{\hat \theta}}

\newcommand{\bbC}{\mathbb{C}}

\newcommand{\bbF}{\mathbb{F}}

\newcommand{\bbN}{\mathbb{N}}

\newcommand{\bbR}{\mathbb{R}}

\newcommand{\bbZ}{\mathbb{Z}}
\newcommand{\bfx}{\mathbf x}

\newcommand{\pfrak}{{\mathfrak p}}       \newcommand{\Pfrak}{{\mathfrak P}}
         \newcommand{\frP}{{\mathfrak P}}

\def\alp{\alpha}

\def\del{\delta}        \def\Del{\Delta}
\def\eps{\varepsilon}

\def\phi{\varphi}

\newcommand{\Quot}{{\rm Quot}}
\begin{document}

\title{Fully Hilbertian Fields}%
\author{Lior Bary-Soroker}%
\author{Elad Paran}
\address{Einstein Institute of Mathematics Edmond J. Safra Campus, Givat Ram, The Hebrew University of Jerusalem, Jerusalem, 91904, Israel}%
\email{barylior@math.huji.ac.il}%
\address{School of Mathematical Sciences, Tel Aviv University, Ramat Aviv, Tel Aviv, 69978, Israel}
\email{paranela@post.tau.ac.il}%

\date{\today}%
\begin{abstract}

We introduce the notion of fully Hilbertian fields, a strictly stronger notion than that of Hilbertian fields. We show that this class of fields exhibits the same good behavior as Hilbertian fields, but for fields of uncountable cardinality, is more natural than the notion of Hilbertian fields. In particular, we show it can be used to achieve stronger Galois theoretic results. Our proofs also provide a step toward the so-called Jarden-Lubotzky twinning principle.
\end{abstract}
\maketitle
% ----------------------------------------------------------------

\section{Introduction}
A field $K$ is called Hilbertian if it satisfies the following
property: for every irreducible polynomial $f(X,Y)\in K[X,Y]$
which is separable in $Y$, there exist infinitely many $a\in K$
for which $f(a,Y)$ is irreducible in $K[Y]$. The name Hilbertian
is derived from Hilbert's irreducibility theorem, which asserts
that number fields are Hilbertian.

Hilbert's original motivation for his irreducibility theorem was
the inverse Galois problem. Assume one wants to realize a finite group $G$ over a Hilbertian field $K$. Then given a regular Galois extension $F/K(x)$ with group $G$, one can use Hilbert's irreducibility theorem to specialize $x$ to $a\in K$ and obtain a Galois extension $\Fgag$ of $K$ with group $G$.  Moreover, there exist infinitely many such specialized fields, linearly disjoint over $K$, obtained by suitable choices of $a \in K$.
Even
nowadays Hilbert's irreducibility theorem remains a central
approach for the inverse Galois problem, see
\cite{MalleMatzat1999,Voelklein1996}. Moreover, this theorem has
numerous other applications in number theory, see e.g.\
\cite{Serre1989}.

It is yet unknown if the inverse Galois problem has a positive solution over a Hilbertian field $K$, or equivalently, if it has a positive solution over $K(x)$. If a field $K$ is ample (e.g.\ $K$ is algebraically closed, PAC, or Henselian) then a theorem of Pop asserts that the inverse Galois problem has a positive solution over $K(x)$. Thus the inverse Galois problem has a positive solution over ample Hilbertian fields.

The cardinality of the set of solution fields $\Fgag$ to each instance of the inverse Galois problem affects the structure of the absolute Galois group of the field $K$. Surprisingly, if $K$ is uncountable, it might happen that a regular Galois extension $F/K(x)$ has only countably many specialized solutions $\Fgag$ (see Example \ref{exa:jarden}). Moreover, in order to study the absolute Galois group, one needs to study not only the possible realizations of finite groups, but also how the obtained fields fit together. This is achieved using finite embedding problems, which give rise to non-regular extensions of $K(x)$.

More precisely, let $f(X,Y)$ be an irreducible polynomial that is separable in $Y$. Let $F=K(X)[Y]/(f(X,Y))$ and let $L$ be the algebraic closure of $K$ in $F$. Each $a$ in the corresponding Hilbert set $H(f)=\{a\in K \mid f(a,Y) \mbox{ irreducible}\}$  provides a specialized field $\Fgag_a$ that contains $L$ with degree $[\Fgag_a:K]=[F:K]$. If $F=L(X)$, then $L=\Fgag_a$ for all $a$, i.e., $L$ is the unique specialized field. In the non-trivial case, Hilbertianity implies the existence of infinitely many specialized fields that are linearly disjoint over $L$.  However, as mentioned above,  even if $K$ is uncountable it is possible that there exist only countably many such fields (and even though the cardinality of $H(f)$ always equals that of $K$).

In this work we study fields which have as much and as distinct as possible specialized fields $\Fgag_a$. That is to say, there exists a subset $A\subseteq H(f)$ of cardinality $|K|$  such that the specialized fields $\Fgag_a$ are  linearly disjoint over $L$ (where $a$ runs over $A$). We call a field with this feature \textbf{fully Hilbertian}. It is important to note that a \emph{countable} field $K$ is Hilbertian if and only if it is fully Hilbertian. In particular, number fields are fully Hilbertian.

The objective of this work is to initiate the study of fully Hilbertian fields. Therefore we give several equivalent definitions of fully Hilbertian fields, construct many fundamental examples of fully Hilbertian fields, and study the behavior of this notion under  extensions. We also show how one can apply this notion and obtain strong Galois theoretic results over fields of large cardinality. We note that as Hilbert's irreducibility theorem found many other applications outside the scope of Galois theory, this notion may be useful for other problems over uncountable fields.

\subsection{Characterizations of fully Hilbertian fields}
A recent characterization of Hilbertian fields by the first author \cite{Bary-Soroker2008IMRN} reduces the Hilbertianity property to absolutely irreducible polynomials. We show that a similar property holds for fully Hilbertian fields:

\begin{theorem}
The following are equivalent for a field $K$.
\begin{enumerate}
\item $K$ is fully Hilbertian.
\item For every absolutely irreducible $f(X,Y)\in K[X,Y]$ that is separable in $Y$ there exist $|K|$ many $a\in H(f)$ and $b_a$ a root of $f(a,Y)$ such that $K(b_a)$ are linearly disjoint over $K$.
\item For every finite \emph{Galois} extension $F/K(x)$ with $L$ the algebraic closure of $K$ in $F$, there exist $|K|$ many $a\in K$ such that $[\Fgag_a:K]=[F:K]$ and all $\Fgag_a$ are linearly disjoint over $L$.
\end{enumerate}
\end{theorem}

\subsection{Finitely generated field extensions}
The most fundamental family of Hilbertian fields is the family of number fields. As mentioned above, since
countable Hilbertian fields are fully Hilbertian, number fields are
fully Hilbertian. The second important family is that of function fields:

\begin{theorem}
\label{IT:function_field}
Let $F$ be a finitely generated transcendental extension of an arbitrary field $K$. Then $F$ is fully Hilbertian.
\end{theorem}

\subsection{Extensions of fully Hilbertian fields}
Hilbertian fields have interesting behavior under algebraic extensions which has been well studied; \cite[Chapter 13]{FriedJarden2005} offers a good treatment of this subject.
We  show that
fully Hilbertian fields exhibit the same behavior.

Any algebraic extension $L/K$ factors into a tower of fields $K \subseteq E \subseteq L$ in which $E/K$ is purely inseparable and $L/E$ is separable.
We study each case separately. First, full Hilbertianity is preserved under
purely inseparable extensions:
\begin{theorem}
\label{IT:purely_iseparable}
Let $E$ be a purely inseparable extension of a fully Hilbertian field $K$. Then $E$ is fully Hilbertian.
\end{theorem}

The case of separable extensions is much more interesting. First,
it is clear that not every extension of a fully Hilbertian (or
Hilbertian) field is Hilbertian. For example, a separably closed
field is not Hilbertian, and hence not fully Hilbertian.
The most general result for Hilbertian fields is Haran's diamond theorem \cite{Haran1999Invent}. We prove an analog of the diamond theorem, and all other permanence criteria for fully Hilbertian fields.

\begin{theorem}
\label{IT:separable}
Let $M$ be a separable extension of a fully Hilbertian field $K$. Then each of the following conditions suffices for $M$ to be fully Hilbertian.
\begin{enumerate}
\item $M/K$ is finite.
\item $M$ is an abelian extension of $K$.
\item
\label{case:weissauer}
 $M$ is a proper finite extension of a Galois extension $N$ of $K$.
\item (The diamond theorem) there exist Galois extensions $M_{1},
M_{2}$ of $K$ such that $M\subseteq M_{1}M_{2}$, but
$M\not\subseteq M_{i}$ for $i=1,2$.
\end{enumerate}
\end{theorem}

(Our full result appears in Section~\ref{sec:sep-ext}.)

In order to prove Theorem~\ref{IT:separable} we identify and exploit a connection between fully Hilbertian fields and so-called semi-free profinite groups. This is a refinement of the \textbf{twinning principle} \cite{JardenLubotzky1992} suggested by Jarden and Lubotzky. They note a connection between
results about freeness of subgroups of free profinite groups and
results about Hilbertianity of separable extensions of Hilbertian
fields.  Furthermore, Jarden and Lubotzky state that the proofs in
both cases have analogies. In spite of that they add that
\textit{it is difficult to see a real analogy between the proofs
of the group theoretic theorems and those of field theory}.

In \cite{Haran1999Group,Haran1999Invent} Haran provides
more evidence to the twinning principle by proving his diamond
theorem in both cases (see also \cite[Theorems~13.8.3 and
25.4.3]{FriedJarden2005}). Haran's main tool in both proofs is
twisted wreath products. (Using twisted wreath product one can
induce embedding problems and then, on the other direction, induce
weak solutions via Shapiro's map. Haran shows that under some
conditions those weak solutions are in fact solutions, i.e.,
surjective.  We refer to this method henceforth as the
\textbf{Haran-Shapiro induction}.)

 In this work we prove all the permanence results for fully Hilbertian fields by reducing them to the group theoretic case.  In fact we show that any group theoretic result obtained via the Haran-Shapiro induction transfers an analogous field theoretic result.
First we explain why a better analogy is between \emph{semi-free profinite}
groups of rank $m$ and \emph{fully Hilbertian} fields of cardinality
$m$\footnote{If one wants to consider only Hilbertian fields, then
our work shows the connection between Hilbertian fields and
profinite groups having the property that any finite split
embedding problem is solvable.}. A semi-free group is, in a sense,
a free group without projectivity, i.e., a semi-free projective
group is free. These groups recently appeared with connection to
Galois theory, see \cite{Bary-SorokerHaranHarbater}.

The idea of the proof is to show that all the constructions in the Haran-Shapiro induction are ``field theoretic''. This is based on the theory of rational embedding problems and geometric solutions developed in the first author's PhD dissertation, see \cite{Bary-SorokerPACext}.

Then all the permanence results about
semi-free profinite groups of \cite{Bary-SorokerHaranHarbater} carry over to fully Hilbertian fields, proving Theorem \ref{IT:separable}.
To the best of our knowledge, this is the first time a reduction from the field theoretic case to its group theoretic counterpart has been made. We hope this is a step towards a rigorous formulation of the twinning principle.

\subsection{Galois theory of complete local domains}

Complete local domains play an important role in number theory and algebraic geometry, and their algebraic properties have been described by Cohen's structure theorem in 1946. However, the Galois theoretic properties of their quotient fields have only recently been understood. The first result in this direction has been made in \cite{HarbaterStevenson2005}, where
Harbater-Stevenson essentially prove that the absolute Galois group of $K=K_0((X,Y))$ is semi-free of rank $|K|$, for an arbitrary base field $K_0$ (see also \cite{Bary-SorokerHaranHarbater}). The general case was treated independently by Pop \cite{Pop2009} and the second author \cite{Paran2010}:

\begin{theorem}
\label{IT:com-krull-dom}
Let $K$ be the quotient field of a complete local domain of
dimension exceeding 1. Then $\gal(K)$ is semi-free of rank $|K|$.
\end{theorem}

The fields $K$ in the above theorem are ample \cite{Pop2009}. For such fields, we have the following result:

\begin{theorem}
If $K$ is a fully Hilbertian ample field, then $\Gal(K)$ is semi-free of rank $|K|$. \end{theorem}

We prove a stronger results about these, namely their maximal purely inseparable extension is fully Hilbertian.

\begin{theorem}
\label{IT:com-krull-dom-hil}
Let $K$ be the quotient field of a complete local domain of
dimension exceeding 1. Then the maximal purely inseparable extension $K_{ins}$ of $K$ is fully Hilbertian.
\end{theorem}

Note that if $K$ is of characteristic $0$, then $K$ itself is fully Hilbertian. 

Theorem~\ref{IT:com-krull-dom-hil} provides a new proof of Theorem~\ref{IT:com-krull-dom} since $\gal(K) = \gal(K_{ins})$. However, we note that Theorem \ref{IT:com-krull-dom-hil} is stronger than Theorem~\ref{IT:com-krull-dom}, since there exist ample fields   of any characteristic with semi-free absolute Galois group that are not fully Hilbertian (Remark~\ref{rem:free-no-hil}).

We do not know in general whether $K$ is fully Hilbertian or not. 
Our key result states that there is $A\subseteq K$ of independent irreducible specializations provided that all the ramification points are separable over $K$. 

To prove Theorem~\ref{IT:com-krull-dom-hil} we use the Hilbertianity of $K$. Using the
density theorem of Hilbertian sets, we find irreducible
specializations with the following nice feature. The geometric
inertia groups of the function fields are ``specialized'' to
inertia groups of valuations of $K$. 

During the Oberfolwach workshop on the arithmetic of fields after reporting this work, we learned from Pop, that he uses a similar approach to prove Theorem~\ref{IT:com-krull-dom}.

\subsection*{Acknowledgment}
We wish to thank Arno Fehm, Dan Haran, Moshe Jarden, and Franz-Viktor Kuhlmann for useful discussions about this work.

\section{Embedding problems, Hilbertian fields, and fully Hilbertian fields}
\subsection{Fully Hilbertian fields in terms of polynomials}
Recall that, for an irreducible polynomial $f(X,Y)\in K[X,Y]$ that is separable in $Y$, the Hilbert set $H(f)\subseteq K$ is the set of all $a\in K$ such that $f(a,Y)$ is irreducible. It will be convenient sometimes to restrict to the co-finite subset $H'(f)\subseteq H(f)$ of all $a\in H(f)$ for which $\deg f(a,Y) = \deg_Y f(X,Y)$ and $f(a,Y)$ separable.
By definition $K$ is Hilbertian if and only if all Hilbert sets are non-empty, and hence
infinite. In fact the cardinality of each Hilbert set must be as
large as possible.

\begin{proposition}[Jarden]\label{prop:largeHilbertsets}
Let $K$ be a Hilbertian field and $H = H(f)$ a Hilbert set over $K$.
Then $|H| = |K|$.
\end{proposition}

\begin{proof}
If $K$ is countable the result is trivial.

Assume $m = |K| > \aleph_0$. Let $\{t_i \mid i\leq m\}$ be a
transcendence basis of $K$ over its prime field $K_0$. Define a
valuation $v$ on $K_0(t_i \mid i\leq m)$ inductively such that
$v(t_1) = 1$ and
\[
 j< i \longrightarrow v(t_j) > v(t_i).
\]
We have  $v(t_i - t_j) = v(t_i) \leq  v(t_1) = 1$, for all $j< i
\leq m$. Extend $v$ to $K$ arbitrarily.

For each $i<m$ let $B_i = \{ x\in K \mid  v(x-t_i) > 1\}$ be an open
ball. If $x\in B_i$, then $v(x) = v(x-t_i + t_i) = v(t_i)$. This
implies that these balls are disjoint. The proof is now done, since
$H$ is dense in the topology defined by $v$ \cite[Lemma~4.1]{GeyerJarden1975}.
\end{proof}

Let us now focus on the fields generated by a root of $f(a,Y)$ where $a$ varies over $H(f)$.
The last result asserts that $|H(f)|=|K|$, hence it is somewhat surprising that the cardinality of the set of fields generated by  a root of $f(a,X)$, $a\in H(f)$ can be smaller than $|K|$. The following example gives a Hilbertian field $K$ such that $|K|>\aleph_0$, but the cardinality of the set of specialized fields is countable.

\begin{example}[Jarden]\label{exa:jarden}
Let $K$ be a pseudo algebraically closed (PAC) field with free
absolute Galois group of rank $> \aleph_0$. Then $|K|>\aleph_0$.
Now the free profinite group $G$ of countable rank is a subgroup
of $\gal(K)$; let $L$ be its fixed field. Then $\gal(L) =G$.
By \cite[Corollary~11.2.5]{FriedJarden2005} $L$ is PAC. Clearly $|L|=|K|$. Furthermore,
$L$ is Hilbertian, as an $\omega$-free PAC field
\cite[Corollary~27.3.3]{FriedJarden2005}.

Every finite separable extension of $L$ corresponds to an open subgroup of $G$. But the rank of $G$
is countable; hence $G$ has only countably many open subgroups, and thus there are only countably many fields generated by roots of $f(a,Y)$, $a\in H(f)$.
\end{example}

We will use the following notation notation: Let $K$ be a field and fix a separable closure $K_s$ of $K$. Consider an irreducible polynomial $f(X,Y)\in K[X,Y]$ that is separable in $Y$ (i.e.\ $\frac{\partial f}{\partial Y} \neq 0$). Let $E=K(X)[Y]/(f(X,Y))$. Then $E/K(X)$ is a finite separable extension.
Let $L = E\cap K_s$ be the algebraic closure of $K$ in $E$.

For every $a\in H'(f)$ there is a unique prime $\pfrak$ of $E/L$ lying above $(X-a)$.
However the residue field $E_a$ of $E$ under $\pfrak$ is not unique.
Since we already fixed $E$ and $K_s$, and hence $L$, $E_a$ is unique up to an element $\sigma\in \gal(L)$. In any rate, $L\subseteq E_a$ independent of the choice of $E_a$.
We define fully Hilbertian fields to have as much as possible specialized fields that are linearly disjoint over $L$.

\begin{definition}
\label{def:ful-hil-gen}
A field $K$ is called fully Hilbertian if for every $f(X,Y)$ the following property holds. Let $E=K(X)[Y]/(f(X,Y))$ and $L = E\cap K_s$. Then there exists a set $A\subseteq H(f)$ of cardinality $|A| = |K|$ and a set of specialized fields $\{E_a \mid a\in A\}$ that are linearly disjoint over $L$.
\end{definition}

\begin{remark}
If $E \cong L(X)$, then $E_a = L$, for all $a$. However, formally, $L$ is linearly disjoint of $L$ over $L$. So in this trivial case, there always exists a set $A$ as in the definition.
\end{remark}

\begin{remark}
Definition~\ref{def:ful-hil-gen} is the most na\"{\i}ve definition. The definition is not canonical, in the sense that there are many choices of the fields $E_a$. This will be solved below, when we give a characterization of fully Hilbertian fields in terms of Galois extensions $E/K(X)$, or more accurately, in terms of rational embedding problems.
\end{remark}

We give a weaker definition of a `regular fully Hilbertian' field. Later we prove that in fact it is the same notion as fully Hilbertian, cf.\ \cite{Bary-Soroker2008IMRN} for the classical case.

\begin{definition}
\label{def:ful-hil-abs-irr}
A field $K$ is called \textbf{regular fully Hilbertian} if for any absolutely irreducible polynomial $f(X,Y)$ there exists a set $A\subseteq H(f)$ of cardinality $|A|=|K|$ and for each $a\in A$, a root $b_a$ of $f(a,X)$ such that the fields $K(b_a)$, $a\in A$ are linearly disjoint over $K$.
\end{definition}

\begin{remark}
It is natural to consider Galois fully Hilbertian fields -- fields having the strong specialization property for every irreducible polynomial $f(X,Y)$ separable and Galois in $Y$. before doing so, we introduce some useful terminology.
\end{remark}

\subsection{Embedding problems}
The following notions of rational embedding problem and geometric solutions are essential in this work, cf.\ \cite{Bary-SorokerPACext}.
Recall that an \textbf{embedding problem} for a field $K$ is a
diagram
\[
\xymatrix{%
    &\gal(K) \ar[d]^{\nu}\\
B\ar[r]^{\alpha}
    &A,
}%
\]
where $B,A$ are profinite groups and $\alpha,\nu$ are (continuous)
epimorphisms. The embedding problem is \textbf{finite} (resp.\
\textbf{split}) if $B$ is finite (resp.\ $\alpha$ splits). It is \textbf{non-trivial} if $\alpha$ is not isomorphism, or equivalently, $\ker \alpha\neq 1$.

A \textbf{weak solution} of $(\nu,\alpha)$ is a homomorphism
$\theta\colon \gal(K) \to B$ such that $\theta\nu = \alpha$. If
$\theta$ is also surjective, we say that $\theta $ is a
\textbf{proper solution}, or in short \textbf{solution}.

We can always replace $A$ with $\gal(L/K)$, where $\gal(L) = \ker
\nu$ and $\nu$ with the restriction map to get an equivalent
embedding problem.

In what follows we show that $B$ can also be considered as a
Galois group of some geometric object. In fact, the interesting object of study are the geometric embedding problems.

\begin{definition}
Let $E$ be a finitely generated regular transcendental extension of $K$, let $F/E$
be a Galois extension, and let $L = F\cap K_s$. Then the restriction
map $\alpha \colon \gal(F/E)\to \gal(L/K)$ is surjective, since
$E\cap K_s = K$. Therefore
\begin{equation}\label{eq:geometricep}
(\nu\colon \gal(K) \to \gal(L/K), \alpha\colon \gal(F/E)\to
\gal(L/K))
\end{equation}
is an embedding problem for $K$. We call such an embedding problem
a \textbf{geometric embedding problem}.

If $E=K(x_1,\ldots,x_d)$ is a field of rational functions over $K$
then we call \eqref{eq:geometricep} a \textbf{rational embedding
problem}.
\end{definition}

\begin{remark}
In the definition the transcendence degree $d$ of $E/K$ can be
arbitrary. However, in the setting of this work, it suffices to consider only the case $d=1$. This observation is based
on the Bertini-Noether lemma and the Matsusaka-Zariski theorem, see
\cite[Proposition~13.2.1]{FriedJarden2005} or
\cite[Lemma~2.4]{Bary-Soroker2008IMRN}.
\end{remark}

\begin{remark}
In \cite[Lemma~11.6.1]{FriedJarden2005} it is shown (in different terminology)
that every finite embedding problem for a field $K$ is equivalent to a geometric
embedding problem.

We shall see below that if $K$ is infinite, then any rational embedding problem has a weak solution. Hence the existence of a weak solution is necessary for an embedding problem to be  equivalent to a rational embedding problem. A central result in Galois theory of Pop \cite{Pop1996} asserts that if $K$ is ample, then every finite split embedding problem over $K$ is regularly solvable, and in our terminology, equivalent to a rational embedding problem.

In fact the regularity of finite split embedding problems immediately implies that the above  necessary condition, i.e., existence of a weak solution, also suffices for an embedding problem over an ample field to be regularly solvable. See Lemma~\ref{lem:basic-rat-eps}.
\end{remark}

The next natural step is to consider solutions to geometric embedding problems that come from field theory, i.e., are defined by places. Before giving the definition, we introduce additional notation. Let $E/K$ be an extension of fields and let $\phi$ be a place of $E$. We say that $\phi$ is a $K$-place if $\phi|_{K} = \id_K$. We denote by $\Egag_\phi$ (or simply $\Egag$ if there is no risk of confusion) the residue field of $E$ under $\phi$.

\begin{definition}\label{def:geometricsolution}
Consider a geometric embedding problem \eqref{eq:geometricep}. Let
$\phi$ be a $K$-place of $E$ and let $\Phi$ be an extension of $\phi$ to an $L$-place of $F$.
Assume that
\begin{enumerate}
\item
\label{def:geo-sol-a}
$\Egag = K$ and
\item
\label{def:geo-sol-b}
$\Phi/\phi$ is unramified.
\end{enumerate}
Then the decomposition group $D = D_{\Phi/\phi}$ is
isomorphic to $\gal(\bar F/K)$ (for a proof, apply \cite[Lemma~6.1.4]{FriedJarden2005} to the valutation rings of $\Phi/\phi$). Thus we have a
canonical map $\Phi^*\colon \gal(K) \to \gal(F/E)$ whose image is
$D$. Clearly $\Phi^* \nu = \alpha$, since $\Phi^*$ respects the
restriction maps (recall that $\nu$ is a restriction map). In other words $\Phi^*$ is a weak solution of \eqref{eq:geometricep}.

Let $\theta\colon \gal(K) \to \gal(F/E)$ be a weak solution of \eqref{eq:geometricep}. We say that $\theta$  is \textbf{geometric}, if there exists $\Phi/\phi$ satisfying \eqref{def:geo-sol-a} and \eqref{def:geo-sol-b} for which $\theta = \Phi^*$.
\end{definition}

The following lemma shows that if we take `larger' or `smaller' embedding problems from a given rational embedding problem, the new ones are also rational.

\begin{lemma}
\label{lem:basic-rat-eps}
Let  $(\nu \colon \gal(K)\to \gal(L/K), \alpha\colon \gal(F/K(x)) \to \gal(L/K))$ be a rational embedding problem for a field $K$.
\begin{enumerate}
\item If $\alpha=\beta\gamma$, for some epimorphisms $\gamma\colon \gal(F/K(x)) \to B$, $\beta\colon B\to \gal(L/K)$, then $(\nu,\beta)$ is rational.
\item
Let $\Phi^*$ be a geometric weak solution of $(\nu,\alpha)$. Then $\gamma \Phi^*$ is a geometric weak solution of $(\nu,\beta)$.
\item
\label{case:fib-pro-ep-is-rat}
 Let $N$ be a Galois extension of $K$ that contains $L$, let $\mu\colon \gal(K)\to \gal(N/K)$ be the restriction map, let $C = \gal(F/K(x))\times_{\gal(L/K)} \gal(N/K)$ be the fiber product, and let $\delta\colon C \to \gal(N/K)$ and $\epsilon\colon C\to \gal(F/K(x)$ be the quotient maps. Then there exists a natural isomorphism $C\cong \gal(FN/K(x))$ under which $\delta,\epsilon$ are the restriction maps
. In particular, $(\mu,\delta)$ is rational.
\item
\label{case:fib-pro-geo-sol} Let $\Phi$ be an $N$-place of $FN$
which is unramified over $K(x)$. If $\theta=\Phi^*$ is a weak
geometric solution of $(\mu,\beta)$, then $\epsilon \theta
=(\Phi|_{F})^*$. In particular, $\epsilon \theta$ is a weak
geometric solution of $(\nu,\alpha)$.
\item
\label{case:fin-pro-geo-sol-split}
If $F/K(x)$ is finite and $K$ infinite, then there exists a finite separable extension $N_0/K$ such that if $N_0\subseteq N$, then the embedding problem $(\mu,\beta)$ splits.
\end{enumerate}
\end{lemma}

\begin{proof}
Let $F_0$ be the fixed field of $\ker \gamma$. Then $\gamma$ induces an isomorphism $\gammagag\colon \gal(F_0/K(x))\to B$. Thus, the following commutative diagram (recall that $\alpha$ is the restriction map) implies that $(\nu,\beta)$ is rational.
\[
\xymatrix{
\gal(F/K(x))\ar[r]^-{\res}\ar[dr]^-{\gamma}
    &\gal(F_0/K(x))\ar[d]^{\gammagag}\ar[dr]^{\res}\\
    &B\ar[r]^{\beta}
        &\gal(L/K)
}
\]
Now under these identifications, $\gamma \Phi^* = (\Phi|_{F_0})^*$, and hence is a geometric weak solution.

To see \eqref{case:fib-pro-ep-is-rat}, recall the natural isomorphism
\[
\gal(FN/K(x))  \cong \gal(F/K(x)) \times_{\gal(L/K)} \gal(N/K) = C
\]
is given by $\sigma \mapsto (\sigma|_{F}, \sigma|_{N})$.

Part~\eqref{case:fib-pro-geo-sol} is immediate since the correlation $\Phi \mapsto \Phi^*$ respects restrictions, and since $\epsilon$ is the restriction map.

To prove \eqref{case:fin-pro-geo-sol-split}, choose $a\in K$ such that $(x-a)$ is unramified in $F$. Extend the specialization $x\mapsto a$ to a place $\phi$ of $F/L$. Let $N_0$ be the residue field of $F$ under $\phi$. Then $\phi^*$ is a geometric weak solution of $(\nu,\alpha)$ that factors via $\gal(N_0/K)$. Choose $N$ as in \eqref{case:fib-pro-ep-is-rat} with $N_0\subseteq N$. We need to show that $\beta\colon C\to \gal(N/K)$ splits.

Indeed, let $\beta'\colon \gal(N/K) \to C$ be defined as follows. For each $\sigma\in \gal(N/K)$ we write $\phi^* (\sigma) = \phi^*(\sigma')$, where $\sigma'\in \gal(K)$ is an extension of $\sigma$ to $K_s$. Since $\phi^*$ factors via $\gal(N_0/K)$ it also factors via $\gal(N/K)$, and hence $\phi^*(\sigma)$ is well defined. Now we set $\beta'(\sigma) = (\phi^*(\sigma), \sigma)$. Clearly $\beta'$ is a section of $\beta$, as needed.
\end{proof}

\subsection{Characterization of fully Hilbertian fields}
We start by formulating the Hilbertianity property in terms of geometric solutions of rational embedding problems.

\begin{proposition}\label{prop:hilbertian-ep}
A field $K$ is Hilbertian if and only if every rational
embedding problem is geometrically solvable.
\end{proposition}

\begin{proof}
We can assume that $K$ is infinite.

Assume that every rational embedding problem is geometrically
solvable. Let $f(X,Y)\in K[X,Y]$ be an irreducible polynomial that
is separable in $Y$ and with a leading coefficient $a(X)\neq 0$.
Choose some finite set $A\subseteq K$ with more elements than the
numbers of roots of $a$.

Let $x$ be a variable, take $F$ to be the splitting field of
$\prod_{\alpha\in A} f(x+\alpha ,Y) f(\frac{1}{x}+\alpha,Y)$ over
$K(x)$, and let $L=F\cap K_s$. Let $\Phi^*$ be a geometric solution of the
rational embedding problem
\[
(\nu\colon \gal(K) \to \gal(L/K), \alpha\colon \gal(F/K(x))\to
\gal(L/K)).
\]

Assume without loss of generality that $\beta = \Phi(x) \in K$
(otherwise we repeat the following argument with $\beta=
\Phi(\frac{1}{x})$). Since $|A|$ is bigger than the number of roots
of $a(x)$, there exists $\alpha\in A$ for which $a(\beta+\alpha)\neq
0$. Let $y\in F$ be a root of $f(x+\alpha,Y)$ and let $\gamma =
\Phi(y)$. Then $\gamma$ is finite. But
since $\Phi^*$ is surjective, $[F:K(x)] = [\Fgag:K]$. In particular
\[
\deg_Y f(x+\alpha,Y) = [K(x,y):K(x)] = [K(\gamma),K],
\]
and thus $f(\beta+\alpha, Y)$ is irreducible (recall that
$f(\beta+\alpha,\gamma)=0$).

Next assume that $K$ is Hilbertian and consider a rational finite
embedding problem
\[
(\nu\colon \gal(K) \to \gal(L/K), \alpha\colon \gal(F/K(x))\to
\gal(L/K)).
\]
Let $f(x,Y)\in K[x,Y]$ be an irreducible polynomial that is
separable in $Y$, a root of which generates $F/K(x)$. Let $0\neq g(x)
\in K[x]$ be the product of the discriminant of $f$ over $K(x)$ and
its leading coefficient. Choose an irreducible specialization
$x\mapsto \alpha\in K$ of $f(x, Y)$ such that $g(\alpha)$ does not
vanish.

Extend this specialization to an $L$-place of $F$. Then $\Phi$ is
unramified over $K(x)$. Moreover, since
$f(\alpha,Y)$ is irreducible, we have that $[\Fgag:K] = [F:K(x)]$.
Thus the decomposition group equals $\gal(F/K(x))$, i.e., $\Phi^*$
is surjective.
\end{proof}

\begin{remark}
\label{rem:free-no-hil}
In the above result, it is necessary that the solutions are
geometric. Indeed, there exists a non-Hilbetian field $K$ whose
absolute Galois group is $\omega$-free (i.e., every finite
embedding problem is solvable): Let $K_0 = \bbC(x)$, then
$\gal(K_0)$ is an infinitely generated free group. The restriction
map $\alpha\colon \gal(K_0((x)) ) \to \gal(K_0)$ splits
(\cite[Corollary~4.1(e)]{GeyerJarden1988}). Let $K$ be the fixed
field of the image of a section of $\alpha$. Then $\gal(K) \cong
\gal(K_0)$ is free of infinite rank. But $K$ is Henselian as an
algebraic extension of the Henselian field $K_0((x))$. By
\cite[Lemma~15.5.4]{FriedJarden2005} $K$ is not Hilbertian.
\end{remark}

We wish to characterize fully Hilbertian fields in terms of geometric solutions of rational embedding problems. In the meanwhile we shall define \emph{Galois fully Hilbertian} fields to be fields with the
proper amount of geometric solutions to a rational embedding problem. Moreover we shall require that the solutions are independent in the following sense. Then we shall prove that Galois fully Hilbertian fields, regular fully Hilbertian fields and fully Hilbertian fields are the same.

\begin{definition}
We call a family of solutions $\{\theta_i\mid i\in I\}$ of a non-trivial
embedding problem \eqref{eq:geometricep} \textbf{independent} (or
\textbf{linearly disjoint}) if the solution fields, i.e.\ the fixed fields of $\ker \theta_i$, 
are linearly disjoint over $L$.
\end{definition}

\begin{remark}
The name derives from the group theoretic analog, cf.\
\cite[Definition~2.3]{Bary-SorokerHaranHarbater}.
\end{remark}

Before continuing, we strengthen Lemma~\ref{lem:basic-rat-eps} and show that independency of solutions is preserved under mild changes of the embedding problems.

\begin{lemma}
\label{lem:basic-rat-eps_2}
In the notation of Lemma~\ref{lem:basic-rat-eps} the following extra assertions hold:
\begin{enumerate}
\item Let $\{\Phi_i^* \mid i \in I\}$ be a family of independent geometric solutions of $(\nu,\alpha)$. Then $\{\gamma \theta_i \mid i\in I\}$ is a family of independent geometric solutions of $(\nu,\beta)$.
\item
 If $\{\theta_i \mid i\in I\}$ is a
set of independent geometric solutions of $(\mu,\beta)$, then
$\{\epsilon \theta_i \mid i\in I\}$ is a set of independent
geometric solutions of $(\nu,\alpha)$.
\end{enumerate}

\end{lemma}

\begin{proof}
We start with the first case. Let $\Fgag_i$ be the residue field of $\Phi_i$. By assumptions, $\{\Fgag_i \mid i\in I\}$ is a linearly disjoint set over $L$. Let $\Egag_i$ be the residue field of $F_0$ under $\Phi_i|_{F_0}$. (Recall that $B\cong \gal(F_0/K(x))$.) Then $\Egag_i\subseteq \Fgag_i$, and hence $\{\Egag_i\mid i\in I\}$ is a linearly disjoint set over $L$. Thus, by definition, $\gamma\Phi^* = (\Phi_i|_{F_0})^*$ are independent solutions of $(\nu,\beta)$.

Next we prove the second case. Denote by $G,B,A$ the groups $\gal(F/K(x))$, $\gal(N/K)$, $\gal(L/K)$, respectively.
By \cite[Lemma~2.5]{Bary-SorokerHaranHarbater}, for any $i_1,\ldots, i_n\in I$,  the image of $\theta_{i_1}\times\cdots\times \theta_{i_n}$ is the fiber product $C^n_B$ of $n$ copies of $C$ over $B$.
Hence by \cite[Lemma~2.4]{Bary-SorokerHaranHarbater} the image of the corresponding $\epsilon\theta_{i_1}\times\cdots\times \epsilon\theta_{i_n}$ equals the fiber product $G^n_A$ of $n$ copies of $G$ over $A$, and thus $\{\epsilon\theta_i \mid i\in I\}$ is an independent set of solutions, by \cite[Lemma~2.5]{Bary-SorokerHaranHarbater}. By Lemma~\ref{lem:basic-rat-eps} these solutions are geometric.
\end{proof}

\begin{definition}
A field $K$ is called \textbf{Galois fully Hilbertian} if every non-trivial rational
finite embedding problem for $K$ has $|K|$ independent geometric solutions.
\end{definition}

\begin{remark}
Note that the definition of Galois fully Hilbertian is equivalent to saying that the property of fully Hilbertian fields holds for all irreducible $f(X,Y)\in K[X,Y]$ that are Galois over $K(X)$.
\end{remark}

\begin{remark}
Finite fields are not Galois fully Hilbertian, since the absolute Galois group of a finite field is pro-cyclic.
\end{remark}

The requirement on the solutions in the definition of Galois fully Hilbertian fields can be weakened, as we show in the next lemma.

\begin{lemma}\label{lem:splitep}
A field $K$ is Galois fully Hilbertian if and only if every non-trivial rational finite split
embedding problem for $K$ has $|K|$ pairwise-independent geometric
solutions.
\end{lemma}

\begin{proof}
We may assume that $K$ is infinite. Consider a rational embedding problem
\[
\xymatrix{
    &\gal(K)\ar[d]^{\nu}\\
\gal(F/K(x))\ar[r]^{\alpha}
    &\gal(L/K).
}
\]
Let $N_0$ be as in Lemma~\ref{lem:basic-rat-eps}\eqref{case:fin-pro-geo-sol-split}. Choose some finite Galois extension $N/K$ that contains $LN_0$. Then Lemma~\ref{lem:basic-rat-eps_2} implies that it suffices to find a set of $m$ independent geometric solutions of the rational finite split embedding problem $(\mu,\beta)$, where $\mu\colon \gal(K)\to \gal(N/K)$ is the restriction map, and
\[
\beta\colon \gal(F/K(x))\times_{\gal(L/K)}  \gal(N/K) \to \gal(N/K)
\]
is the quotient map of the fiber product.

By assumption there exists a set of pairwise-independent geometric solutions $\{\theta_i\mid i\in I\}$ of $(\mu,\beta)$ with $|I|=|K|$. By \cite[Proposition~3.6]{Bary-SorokerHaranHarbater} there exists a subset $I'\subseteq I$ of the same cardinality $|I'|=|K|$ for which  $\{\theta_i\mid i\in I'\}$ is an independent set of geometric solutions, as needed.
\end{proof}

Standard set theoretic considerations show that it suffices to solve infinite embedding problems of cardinality $< |K|$.

\begin{lemma}\label{lem:<m_ep}
For a field $K$ to be Galois fully Hilbertian it suffices that any
rational embedding problem \eqref{eq:geometricep} in which
$\rank\gal(F/K(x)) < |K|$ and $\ker \alpha$ is finite is geometrically solvable.
\end{lemma}

\begin{proof}
Consider a rational finite embedding problem
\[
(\nu\colon \gal(K) \to \gal(L/K), \alpha\colon \gal(F/K(x))\to
\gal(L/K))
\]
for $K$ and write $B = \gal(F/K(x))$ and $A = \gal(L/K)$. In
particular $F$ is regular over $L$.

For each $i<m$ we inductively construct a geometric solution
$\theta_i$ such that the corresponding solution field $N_i$ is
linearly disjoint of the compositum $\prod_{j<i} N_j$ over $L$. This
would imply that $\{\theta_i \mid i< m \}$ is a family of independent
geometric solutions, and hence the proof.

Assume we have constructed geometric solutions $\theta_j$ for each
$j < i$ that are independent. Let $N_j$ be the solution field of
$\theta_j$, let $B_j= \gal(N_j/K)$. Then $B \cong B_j$.

Write $N = \prod_{j < i} N_j$ and let $\nu_i \colon \gal(K) \to
\gal(N/K)$ be the restriction map. Since the $\{\theta_j\mid j<i\}$ are
independent it follows that $\gal(N/K)$ is canonically isomorphic to
the fiber product $\prod_A B_j$ of $B_j$ over $A$, where
$j<i$. The isomorphism is given by $\sigma \mapsto (
(\sigma|_{N_j}))_{j<i}$ \cite[Lemma~2.5]{Bary-SorokerHaranHarbater}.

Moreover, since $F$ is regular over $L$, it is linearly disjoint of
$N$, and hence of $N(x)$ over $L(x)$. Hence
\[
\gal(FN/K(x)) \cong \gal(F/K(x))\times_{\gal(L/K)} \gal(N/K) \cong
B\times_A (\prod_A B_j),
\]
so $\rank \gal(FN/K(x)) \leq i+1 <m$. Let $\alpha_i \colon
\gal(FN/K(x)) \to \gal(N/K)$ be the restriction map.

The assumption gives a geometric solution $\Phi^* \colon \gal(K) \to
\gal(FN/K(x))$ of the rational embedding problem
\[
(\nu_i \colon \gal(K) \to \gal(N/K), \alpha_i \colon \gal(FN/K(x))
\to \gal(N/K)).
\]
Let $M$ and $N_i$ be the residue fields of $FN$ and $F$ under
$\Phi$, respectively, and let $\phi_i = \Phi|_{F}$.

Then $\theta_i = \phi_i^*\colon \gal(K) \to \gal(F/K(x))$ is a
geometric solution. Since the residue field of $N(x)$ is $N$ and
$\Phi^*$ is an isomorphism, we have $[M:N] = [FN : N(x)] = |B|/|A|$.
But $M = NN_i$; so $[NN_i : N] = |B|/|A| =
[N_i:L]$. Therefore $N_i$ is linearly disjoint of $N$ over $L$, as
required.
\end{proof}

We are now ready to prove that the notions of Galois fully Hilbertian fields, regular fully Hilbertian fields, and fully Hilbertian fields are equivalent.

\begin{theorem}
\label{thm:equiv-def}
Let $K$ be a field. The following conditions are equivalent.
\begin{enumerate}
\item $K$ is fully Hilbertian.
\item $K$ is regular fully Hilbertian.
\item $K$ is Galois fully Hilbertian.
\end{enumerate}
\end{theorem}

\begin{proof}
It is clear that full Hilbertianity implies regular full Hilbertianity.

Assume $K$ is regular fully Hilbertian. We prove that $K$ is Galois fully Hilbertian. Let
\[
(\nu\colon \gal(K) \to \gal(L/K), \alpha \colon \gal(F/K(x)) \to \gal(L/K))
\]
be a rational non-trivial finite split embedding problem. By Lemma~\ref{lem:splitep} it suffices to find $|K|$ independent geometric solutions of $(n_0,\alpha_0)$. By Zorn's lemma there exists a  maximal set of geometric independent solutions  $\{ \theta_i \mid i\in I\}$ of $(\nu, \alpha)$. We claim that $|I| = |K|$.
Assume otherwise, i.e., $|I|<|K|$.

Let $N$ be the compositum of the fixed fields of $\ker \theta_i$, $i\in I$. Since $|I|<|K|$ the rank of $\gal(N/K)$ is less than $|K|$. Thus
\begin{equation}
\label{eq:rank<|K|}
\mbox{
there exist at most $|I|$ finite subextensions of $N/K$. }
\end{equation}

By assumption $\alpha$ splits, let $\alpha'\colon \gal(L/K)\to \gal(F/K(x)$ be a section of $\alpha$ and let $E$ be the fixed field of the image of $\alpha'$. Then $E L = F$ and $E\cap L = K$. In particular $E/K$ is regular. Choose a primitive element $y\in E$ of $E/K(x)$ that is integral over $K[x]$. Then $E=K(x,y)$ and the irreducible polynomial $f(x,Y)\in K[x,Y]$ of $y$ over $K(x)$ is monic and absolutely irreducible.

Since $K$ is regular fully Hilbertian there exists a set $A\subseteq H(f)$ of cardinality $|A|=|K|$ and a choice of roots $\ygag_a$ of $f(a,Y)$ such that the fields $E_a=K(\ygag_a)$ are linearly disjoint over $K$. In particular $E_a\cap E_b = K$ for all $a\neq b$.

We claim that there exist infinitely many $a\in A$ such that  $E_a\cap N = K$, and hence $E_a$ and $N$ are linearly disjoint over $K$ (recall that $N/K$ is Galois). Otherwise, we would have that $E'_a = E_a \cap N\neq K$ (where $a$ runs on a co-finite subset of $A$) are $|K|$ finite subextensions of $N/K$. Since
\[
K\subseteq E'_a\cap E'_b \subseteq E_a\cap E_b \subseteq K,
\]
it follows that $E'_a\cap E'_b = K$, hence $E'_a$, $E'_b$ are distinct. This contradicts \eqref{eq:rank<|K|}.

Choose $a\in A$ such that $E_a$ and $N$ are linearly disjoint and if $\phi$ is a place of $F$ that lies over $(x-a)$, then $\Fgag = \overline{EL} = \Egag L = E_a L$. Since $E_a$ and $L$ are also linearly disjoint, we have
\[
[\Fgag:K] = [E_a L:K]=[E_a:K][L:K] = [E:K(x)] [L(x):K(x)] = [F:K(x)],
\]
hence $\phi^*\colon \gal(K) \to \gal(F/K(x))$ is surjective.
Now $E_a,N$ are linearly disjoint over $K$, hence $\Fgag$, $N$ are linearly disjoint over $L$. But this means that $\{\phi^*, \theta_i\}$ is a set of geometric independent solutions, a contradiction to the maximality  of $\{\theta_i\mid i\in I\}$.

It remains to prove that a Galois fully Hilbertian field $K$ is fully Hilbertian.
Let $f(X,Y)$ be an irreducible polynomial that is separable in $Y$. Let $x$ be a variable, and let $y$ be a root of $f(x,Y)$ in some fixed algebraic closure of $K(x)$. Let $E= K(x,y)$. Then $E/K(x)$ is a finite separable extension. Let $L_0=E\cap K_s$.

Choose a Galois extension $F/K(x)$ such that $E\subset F$ and let $L=F\cap K_s$. Then $EL\subseteq F$ and, since $E$ is regular over $L_0$,
\begin{equation}\label{eq:EL:L=E:L_0}
[EL:L] = [E:L_0].
\end{equation}
Since $K$ is Galois fully Hilbertian, it follows that there exists a family $\{\phi_i^*\mid i\in I\}$ of independent geometric solutions of the non-trivial finite rational embedding problem $(\mu \colon \gal(K) \to \gal(L/K), \alpha\colon \gal(F/K(x)) \to \gal(L/K))$, and $|I|=|K|$.
The set of $i\in I$ for which one of the following conditions is not satisfied is finite, hence we can assume that every $i\in I$ satisfies these conditions.
\begin{enumerate}
\item $a_i = \phi_i(x)$ and $b_i = \phi(y_i)$ are finite, for all $i\in I$.
\item $f(a_i,b_i) = 0$.
\item The residue field of $E$ under $\phi_i$ is $\Egag_i=K(b_i)$.
\item The residue field of $EL$ under $\phi_i$ is $\overline{EL}_i = \Egag_i L = L(b_i)$
\end{enumerate}
Since $\phi_i^*$ is surjective for all $i\in I$, we have $[\Fgag_i :K] = [F:K(x)]$, and hence $[\Rgag_i : \Sgag_i]$ for any tower $K(x)\subseteq S\subseteq R\subseteq F$. Hence to finish the proof, it suffices to prove that the $\Egag_i$ are linearly disjoint over $L_0$, for any finite subset $J$ of $I$.

Let $J$ be a finite subset of $I$. Since $\{\phi^*_i\mid i\in J\}$ are independent, the fields $\Fgag_i$, $i\in J$ are linearly disjoint over $L$. In particular $\Egag_i L$ are linearly disjoint over $L$. By \eqref{eq:EL:L=E:L_0} we have
\begin{eqnarray*}
[E:L_0]^{|J|} &=& \prod_{i\in J} [E:L_0] = \prod_{i\in J} [\Egag_i:L_0] \geq [\prod_{i\in J} \Egag_i :L_0] \geq [\prod_{i\in J} (\Egag_i L) : L] \\
	&=& \prod_{i\in J} [\Egag_i L : L] = \prod_{i\in J} [E L : L] = [EL:L]^{|J|} = [E:L_0]^{|J|}.
\end{eqnarray*}
We thus get that all inequalities are equalities, and thus $\Egag_i$ are linearly independent over $L_0$.
\end{proof}

For countable fields, Hilbertianty is the same as full Hilbertianity.
This is an immediate consequence of
Proposition~\ref{prop:hilbertian-ep} and Lemma~\ref{lem:<m_ep}.
\begin{corollary}
\label{cor:countableHilbertian}
Let $K$ be a countable field. Then $K$ is Hilbertian if and only if
$K$ is fully Hilbertian.
\end{corollary}

\subsection{The absolute Galois group of fully Hilbertian fields}
In this section we consider the group theoretic properties of the
absolute Galois group of fully Hilbertian fields.

Recall the definition of semi-free groups \cite{Bary-SorokerHaranHarbater}.

\begin{definition}
A profinite group $\Gamma$ of infinite rank $m$ is called \textbf{semi-free} if every finite split embedding problem for $\Gamma$ has a set of $m$ independent solutions.
\end{definition}

The following conjecture is the analog of Conjecture ``Split EP/$_{K hilb.}$" of \cite{DebesDeschamps1997}.

\begin{conjecture}
\label{conj:semi-free}
The absolute Galois group of a fully Hilbertian field $K$ is semi-free of rank $|K|$.
\end{conjecture}

In \cite{DebesDeschamps1997} D\`ebes and Deschamps show that ``Split EP/$_{K hilb.}$" follows from the conjecture they call ``Split EP/$_{k(t)}$." We show that a slightly stronger property implies Conjecture~\ref{conj:semi-free}.

\begin{proposition}
\label{prop:Cond-semi-free}
Let $K$ be a fully Hilbertian field and assume that any finite split embedding problem for $K$ is equivalent to a rational embedding problem. Then $\gal(K)$ is semi-free of rank $|K|$.
\end{proposition}

\begin{proof}
Any split embedding problem over $K$ is equivalent to
a rational embedding problem. Therefore, by definition there exist
$|K|$ independent geometric solutions.
\end{proof}

Recall that a field $K$ is called \textbf{ample} (or \textbf{large}) if every curve defined over $K$
with a simple $K$-rational point has infinitely many $K$-rational points. In \cite{Pop1996}, Pop proves that the assumption of Proposition~\ref{prop:Cond-semi-free} holds for an ample field, see also \cite{HaranJarden1998}.
Therefore we get:

\begin{corollary}
\label{cor:f-Hil+ample=semi-free}
Let $K$ be a fully Hilbertian ample field. Then $\gal(K)$ is semi-free of rank $|K|$.
\end{corollary}

A stronger property than being ample is being PAC: A field $K$ is PAC if every curve defined over $K$ has infinitely many $K$-rational points.
The next result extends Roquette's Theorem \cite[Corollary~27.3.3]{FriedJarden2005}.

\begin{proposition}
Let $K$ be a PAC field. Then $K$ is fully Hilbertian if and only if
$\gal(K)$ is free of rank $|K|$.
\end{proposition}

\begin{proof}
Assume $K$ is fully Hilbertian. Then Corollary~\ref{cor:f-Hil+ample=semi-free} implies that $\gal(K)$ is semi-free of rank $|K|$. But \cite[Theorem~11.6.2]{FriedJarden2005} implies that $\gal(K)$ is projective, and
thus $\gal(K)$ is free of rank $|K|$
\cite[Theorem~3.5]{Bary-SorokerHaranHarbater}.

Assume that $\gal(K)$ is free of rank $|K|$. Then any finite
embedding problem has $m$ independent solutions \cite[Propsition~25.1.6]{FriedJarden2005}. Since $K$
is PAC, any solution of a geometric embedding problem is geometric
\cite[Corollary~3.4]{Bary-SorokerPACext}. Therefore any rational finite embedding problem
for $K$ has $m$ independent geometric solutions, i.e., $K$ is fully
Hilbertian.
\end{proof}

We now consider arbitrary fields. We give a result about embedding problems with abelian kernel.

\begin{proposition}
\label{prop:abelian-ker}
Let $K$ be a fully Hilbertian field of cardinality $m$, $A$ an abelian group,  $L/K$ a finite Galois extension with Galois group $G=\gal(L/K)$. Assume that $G$ acts on $A$. Then the infinite group $A^m \rtimes G$ occurs as a Galois group over $K$. In particular $A^m$ occurs as a Galois group over $K$.
\end{proposition}

\begin{proof}
Let $\mu\colon \gal(K) \to G$ be the restriction map, and $\alpha\colon A\rtimes G \to G$ the quotient map.
By Ikeda's theorem \cite[Proposition 16.4.4]{FriedJarden2005}, the embedding problem $(\mu,\alpha)$ is equivalent to a rational embedding problem. Since $K$ is fully Hilbertian, there exists a set $S$ of $m$ independent solutions.

The fiber product of $m$ copies of $A\rtimes G$ over $G$ is isomorphic to $A^m\rtimes G$ \cite[Lemma~4.3]{Bary-SorokerHaranHarbater} and the corresponding solution
\[
\prod_{\theta\in S} \theta \colon \gal(K) \to A^m\rtimes G
\]
is surjective \cite[Lemma~2.5]{Bary-SorokerHaranHarbater}.
\end{proof}

For any field $K$, the rank of $\gal(K)$ is bounded by $|K|$.
Taking $A=\bbZ/2\bbZ$ in Proposition~\ref{prop:abelian-ker}, we
get the following result:

\begin{proposition}
If $K$ is fully Hilbertian, then the rank of $\gal(K)$ equals $|K|$.
\end{proposition}

\section{Rational Function Fields}
In this section we show that the basic families of Hilbertian
fields --- number fields and function fields --- are fully
Hilbertian. Since for a countable field full Hilbertianity is
equivalent to Hilbertianity, we are left with function
fields over an infinite field. Our proof is based on an explicit
description of elements in Hilbert sets and on the fact that full
Hilbertianity is preserved by finite extensions. The latter assertion  will be
proved below, see Corollary~\ref{cor:finite-ext}.

Let $K_0$ be an infinite field and $K=K_0(t)$ a rational function field over
$K_0$.
Then each Hilbert set contains a set
\[
\{ a + bt  \mid g(a,b)\neq 0\},
\]
for some nonzero polynomial $g(A,B)\in K_0[A,B]$ \cite[Proposition~13.2.1]{FriedJarden2005}. In this section
we use this description in order to show that $K$ is fully Hilbertian.

\begin{theorem}\label{thm:ratff}
Let $K_0$ be a field and let $K=K_0(t)$ be a rational function field.
Then $K$ is fully Hilbertian.
\end{theorem}

\begin{proof}
If $K_0$ is countable, then Hilbertianity is the same as fully
Hilbertianity. Hence $K$ is fully Hilbertian.

Assume that $m = |K_0|$ is infinite. By Lemma~\ref{lem:<m_ep} it
suffices to geometrically solve a rational embedding problem
\[
(\nu\colon \gal(K) \to \gal(L/K), \alpha\colon \gal(F/K(x))
\to\gal(L/K))
\]
under the assumption that $\rank \gal(F/K(x)) = n < m$.

Since the number of finite subextensions of $F/K(x)$ is bounded by
$n$ (and actually equals $n$ if $n$ is infinite), we can order
the finite subextensions $E\neq K(x)$ of $F/K(x)$, say $\{E_i \mid
i\leq n\}$. For each $E_i$ there exists a Hilbert set $H_i$ such that
for each $u\in H_i$ and for each $K$-place $\phi$ of $E_i$ that lies over $(x-u)$ the
degree of the residue field extension equals to $[E_i:K(x)]$. In
particular $\phi$ is unramified in $E_i$.

Let $g_i(A,B)\in K_0[A,B]$ be the nonzero polynomial such that $\{a +bt\mid a,b\in K_0\mbox{ and }g(a,b)\neq 0\} \subseteq
H_i$. Write $C_i(K_0) = \{ (a,b)\in K_0^2 \mid g(a,b) = 0\}$ for the
corresponding curve. Now since $n<m$
\[
K_0^2 \neq  \bigcup_{i\leq n} C_i(K_0).
\]
Let $(a,b)\in K_0^2 \smallsetminus \bigcup_{i\leq n} C_i(K_0)$.
Extend the specialization $x\mapsto a + b t\in K$ to a $K$-place $\phi$ of
$F$.

For each $i\leq n$, $g_i(a,b)\neq 0$, hence $\phi|_{E_i}$ is unramified
over $K(x)$. Therefore $\phi$ is unramified over $K$, and thus the geometric weak
solution $\phi^*\colon \gal(K) \to \gal(F/K(x))$ of $(\nu,\alpha)$
is well defined.

We claim that $\phi^*\colon \gal(K) \to \gal(F/K(x))$ is a solution,
i.e., surjective. To show that it suffices to show that
$D_\phi=\gal(F/K(x))$, where $D_\phi$ is the corresponding decomposition group.
Let $\Delta$ be the decomposition field (i.e.\ the fixed field of
$D_\phi$ in $F$). We need to show that $\Delta = K(x)$.

If $\Delta \neq K(x)$, then there exists a finite extension $E_i$ of $K(x)$
such that $E_i \subseteq \Delta$. Hence the residue field $\bar E_i$
of $E_i$ is $K$. But since $g_i(a,b) \neq 0$ we have $[\bar E_i : K]
= [E_i:K(x)] > 1$.

This contradiction implies that $\Delta = K(x)$, and hence $\phi^*$ is an epimorphism, which concludes the proof.
\end{proof}

\begin{corollary}[Theorem~\ref{IT:function_field}]
Let $F$ be a finitely generated transcendental extension of an arbitrary field $K$. Then $F$ is fully Hilbertian.
\end{corollary}

\begin{proof}
Let $(t_1,\ldots, t_r)$ be a transcendence basis of $F/K$.
Then $F$ is a finite extension of $K_0(t_1, \ldots,
t_r)$. By Theorem~\ref{thm:ratff} the field $K(t_1,\ldots, t_r) =
K(t_1,\ldots,t_{r-1})(t_r)$ is fully Hilbertian. By Corollary~\ref{cor:finite-ext} below,
$F/K(t_1,\ldots, t_r)$ is fully
Hilbertian.
\end{proof}

\section{The twinning principle}
\label{sec:twinning} We mentioned in the introduction the
\emph{twinning principle} formulated by Lubotzky-Jarden. This
principle suggests there should be a connection between results
about free subgroups of a free profinite group and results about
Hilbertian separable extensions of a Hilbertian field. We also
remarked in the introduction that the \emph{Haran-Shapiro
induction} gives a strong evidence to this principle, since it
proves the diamond theorem in both settings, and the proofs are
analogous.

In fact results about Hilbertian fields transfer to results about free profinite groups of countable rank, using Roquette's theorem, see \cite[Weak twinning principle]{JardenLubotzky1992}.

In this section we further study the twinning principle.
We first claim that the connection should be between semi-free profinite groups of rank $m$ and fully Hilbertian fields of cardinality $m$.
\begin{enumerate}
\item The connection is evident from the characterization of fully Hilbertian fields 
as Galois fully Hilbertian fields (Theorem~\ref{thm:equiv-def}).  A semi-free group $\Gamma$ of rank $m$ is a profinite group for which every non-trivial finite split embedding problem has $m$ independent solutions, while a Galois fully Hilbertian field  of cardinality $|K|=m$ is a field for which every non-trivial \textit{rational} embedding problem has $m$ independent \textit{geometric} solutions.
\item The absolute Galois group of a fully Hilbertian field is not projective in general,
and free groups are projective. In the meanwhile semi-free groups are, in a sense, free group without projectivity
(\cite[Theorem~3.5]{Bary-SorokerHaranHarbater}).
\item A fully Hilbertian ample field has a semi-free group rank $|K|$ (Corollary~\ref{cor:f-Hil+ample=semi-free}), and conditionally this is always the case (Proposition~\ref{prop:Cond-semi-free}).
\end{enumerate}
Maybe the most convincing argument for this connection is the following
\begin{enumerate}\addtocounter{enumi}{3}
\item
\label{case:twinning-carry-over}
The proofs about sufficient conditions for a subgroup of a semi-free group to be semi-free carry over to the field theoretic setting.
\end{enumerate}

\begin{remark}
If one wants to consider the more classical setting of Hilbertian fields, then our method gives the connection with the property that each finite split embedding problem has a solution. One can even consider only semi-free groups of countable rank.
\end{remark}

Let us go into more details. First note that there is a certain trade-off here -- in the group
theoretic setting we need to solve all finite split
embedding problems, while in field theory only rational
embedding problems come into the picture. On the other hand, in
the field theoretic setting we need the solution to have an extra property, namely only geometric solutions are considered.

The way we transfer group theoretic proofs into field theory is via the Haran-Shapiro induction. This method uses  fiber products and twisted wreath products in order to induce solutions of embedding problems from the group to its subgroup. So what we need to do is to prove, or essentially to observe, that all constructions of the Haran-Shapiro induction are field theoretic constructions. That is to say, (1) if one starts from a rational embedding problem, then the constructed embedding problems are also  rational (2) the solutions induced by the constructions preserve the property of being geometric.

As an immediate consequence of this, all the permanence properties of semi-free groups carry over to fully Hilbertian fields. In particular we get all cases of Theorem~\ref{IT:separable}.

\subsection{Fiber products and twisted wreath products}
As indicated before in results about semi-free subgroups of
semi-free groups there are two group theoretic constructions
involved, fiber products and twisted wreath products. Let us start
with the more elementary one, fiber products.

Write $\Gamma = \gal(K)$ and consider an embedding problem $(\mu \colon \Gamma \to G, \alpha\colon H\to G)$ for $K$. Let $\Delta$ be a normal subgroup of $\Gamma$ contained in $\ker \mu$. Then $\mu$ factors as $\mu=\mugag \muhat$, where $\muhat\colon \Gamma \to \Ghat:=\Gamma/\Delta$ is the quotient map and $\mugag\colon  \Ghat \to G$ is canonically defined. This data defines a corresponding \textbf{fiber product embedding problem}:
\[
\xymatrix{
    &\Gamma\ar[d]_{\muhat}\ar@/^/[dd]^{\mu}\\
H\times_{G} \Ghat \ar[r]^{\alphahat}\ar[d]^{\beta}
    &\Ghat\ar[d]_{\mugag}\\
H\ar[r]^{\alpha}
    &G.
}
\]
Here $H\times_{G} \Ghat = \{ (h,\ghat)\in H\times \Ghat \mid \alpha(h) = \mugag(\ghat)\}$ and $\alphahat, \beta$ the canonical projections.

A (weak) solution $\thetahat\colon \Gamma\to H\times_G \Ghat$ of the fiber product embedding problem induces the (weak) solution $\theta=\beta\thetahat$ of $(\mu,\alpha)$.

The following result asserts that this construction preserves field theoretic aspects.
\begin{lemma}
\label{lem:fiber_split}
In the above notation, if $(\mu,\alpha)$ is geometric, then so is $(\muhat,\alphahat)$. Moreover, if $\thetahat$ is geometric, so does $\theta$.

More precisely, assume $H = \gal(F/K(x))$, $G=\gal(L/K)$ for $L=F\cap K_S$, and $\alpha,\mu$ are the restriction maps. Let $N$ be the fixed field of $\Delta$. Then $\Ghat = \gal(N/K)$, $H\times_{G} \Ghat = \gal(FN/K(x))$, and $\alphahat,\beta$ are the corresponding restriction maps. Moreover, if $\thetahat = \hat \Pfrak^*$ for some unramified prime extension  $\hat \Pfrak/(x-a)$ of $FN/K(x)$, then $\theta = (\hat\Pfrak\cap F)^*$.
\end{lemma}

\begin{proof}
See \cite[Lemma~2.6]{Bary-SorokerPACext}.
\end{proof}

Now we move to the more sophisticated construction of twisted wreath products.
For the reader's convenience we give a short survey on the Haran-Shapiro induction, starting from the definitions.  See \cite{Bary-SorokerHaranHarbater} for full details and proofs.

\begin{definition}
Let $A$, $G_0 \leq G$ be finite groups. Assume that $G_0$ acts on
$A$ (from the right). Let
\[
\Ind_{G_0}^G(A) = \{ f\colon G\to A\mid f(\sigma\rho) =
f(\sigma)^\rho,\ \forall \sigma\in G, \rho\in G_0\} \cong
A^{(G:G_0)}.
\]
Then $G$ acts on $\Ind_{G_0}^G(A)$ by $f^\sigma(\tau) = f(\sigma
\tau)$. We define the \textbf{twisted wreath product} to be the
semidirect product
\[
A\wr_{G_0} G = \Ind_{G_0}^G(A) \rtimes G,
\]
i.e., an element in $A\wr_{G_0} G$ can be written uniquely as
$f\sigma$, where $f\in \Ind_{G_0}^G(A)$ and $\sigma\in G$. The
multiplication is then given by $(f\sigma) (g\tau) = f
g^{\sigma^{-1}} \sigma\tau$. The twisted wreath product is equipped
with the quotient map $\alpha\colon A\wr_{G_0} G\to G$ defined by
$\alpha(f\sigma) = \sigma$.

The map $\pi \colon \Ind_{G_0}^G(A) \rtimes G_0  \to A\rtimes G_0$
defined by $\pi (f \sigma) = f(1) \sigma$ is called the
\textbf{Shapiro map}. It is an epimorphism, since, for $f\in \Ind_{G_0}^G(A)$ and $\sigma\in G_0$, we have
\[
\pi(f^\sigma)=f(\sigma)=f(1)^\sigma = \pi(f)^\sigma.
\]
\end{definition}

\begin{definition}
A tower of fields
\[
K\subseteq L_0 \subseteq L\subseteq F\subseteq \Fhat
\]
is said to \textbf{realize} the twisted wreath product $A\wr_{G_0} G$ if $\Fhat/K$ is a Galois extension with Galois group isomorphic to $A\wr_{G_0} G$ and the tower of fields corresponds to the subgroups (via taking fixed fields in $\Fhat$)
\[
A\wr_{G_0} G \geq \Ind_{G_0}^G(A)\rtimes G_0\geq  \Ind_{G_0}^G(A) \geq \ker \pi \geq 1.
\]
We note that the extension $F/L_0$ is Galois with group $A\rtimes G_0$, and the restriction map $\gal(\Fhat/L_0)\to \gal(F/L_0)$ is the corresponding Shapiro map of $A\wr_{G_0}G$. For more details see \cite[Remark~13.7.6]{FriedJarden2005} or \cite[Remark~1.2]{Haran1999Invent}.
\end{definition}

Let $\Lambda \leq \Gamma$ be profinite groups. Consider a  finite split embedding problem $(\mu_1\colon \Lambda \to G_1, \beta_1\colon A\rtimes G_1\to G_1)$ for $\Lambda$.  To use the Haran-Shapiro induction, first we choose an open normal subgroup $\Delta$ of $\Gamma$ such that $\Delta\cap \Gamma\leq \ker \mu_1$.
Let $\mu\colon \Gamma\to G:=\Gamma/\Delta$ be the quotient map and $\mu_0=\mu|_{\Lambda}\colon \Lambda\to G_0:=\Lambda/\Lambda\cap\Delta$.  Then $\mu_1$ factors as $\mu_1 = \mugag_1\mu_0$. Moreover we have the following embedding problems
\begin{equation}
\label{eq:EPinHSI}
\xymatrix@C=40pt{
    &\Lambda\ar[d]_{\mu_0} \ar@/^10pt/[ddd]^{\mu_1}\\
\Ind_{G_0}^G(A)\rtimes G_0 \ar[r]^{\beta} \ar[d]^{\pi}
    &G_0\ar@{=}[d]\\
A\rtimes G_0 \ar[r]^{\beta_0}\ar[d]^{\rho}
    &G_0\ar[d]_{\mugag_1}\\
A\rtimes G_1 \ar[r]^{\beta_1}
    &G_1.
}
\qquad
\xymatrix@C=40pt{
    &\Gamma\ar[d]^{\mu}\\
A\wr_{G_0} G \ar[r]^{\alpha}
    &G.
}
\end{equation}
Here $\beta$ is the restriction of $\alpha$ to $\Ind_{G_0}^G(A)\rtimes G_0$, $\pi$ is the Shapiro map, $A\rtimes G_0 \cong (A\rtimes G_1)\times_{G_1} G_0$, and $\rho,\beta_0$ are the projection maps.
Note the the lower square in the left diagram is the fiber product embedding problem.

Now, if $\theta\colon \Gamma\to A\wr_{G_0} G$ is a weak solution, then $\theta_0 = \rho\pi \theta|_{\Lambda}$ is a weak solution of the original embedding problem $(\mu_1,\beta_1)$. The problem here is that if $\theta$ is a solution, i.e., surjective, it does not imply that $\theta_0$ is a solution. So we are lead to Haran's contribution to the method: Under some conditions on the subgroup $\Lambda$ (e.g., open, $\Gamma'\leq \Lambda$, or most generally, contained in a diamond) the Haran-Shapiro induction gives that, taking $\Delta$ small enough, or equivalently, $G$ large enough, we have
\[
\theta \mbox { surjective} \Longrightarrow \theta_0  \mbox { surjective}.
\]
Moreover, if $\{\theta_i \mid i\in I\}$ is a family of pairwise-independent solutions of $(\mu,\alpha)$, then the induced solutions $\{\theta_{i0}  \mid i\in I\}$ of $(\mu_1, \beta_1)$ are also pairwise-independent.

Now we show that these constructions carry over to the field theoretic setting.
For this assume that $M$ is a separable extension of an \emph{infinite} field $K$, $\Gamma = \gal(K)$, and $\Lambda = \gal(M)$. Also denote by $L$ the fixed field of $\Delta$, i.e., $G = \gal(L/K)$.

\begin{lemma}
\label{lem:wre_prod}
If $(\mu_1,\beta_1)$ is a rational embedding problem for $M$, then there exists a finite separable extension $K'/K$ such that  if $K'\subseteq L$, then all the embedding problems in \eqref{eq:EPinHSI} are rational, and all maps are restriction maps. Moreover, if a weak solution $\theta$ of $(\mu,\alpha)$ is geometric, then the induced weak solution $\theta_0 = \rho\pi\theta|_{\gal(M)}$ of $(\mu_1,\beta_1)$ is also geometric.

More precisely, assume that $A\rtimes G_1 = \gal(F_1/M(x))$, $G_1 = \gal(N_1/M)$ for $N_1 = F_1\cap M_s$, and $\beta_1,\mu_1$ are the restriction maps. Then there exists a finite separable extension $K'/K$ such that for every finite Galois extension $L$ of $K$ that contains $K'$ we have:
\begin{enumerate}
\item
\label{lem:wre_prod_a}
$N_1 \subseteq ML$.
\end{enumerate}
Let $L_0=L\cap M$ and $G_0 = \gal(ML/L) \cong \gal(L/L_0)$.
\begin{enumerate}\addtocounter{enumi}{1}

\item
\label{lem:wre_prod_b}
There exist $x_0\in L_0(x)$ such that $L_0(x_0) = L_0(x)$ and  a tower of fields 
\[
K(x_0) \subseteq L_0(x) \subseteq L(x) \subseteq F \subseteq \Fhat
\]
that realizes $A\wr_{G_0} G$. In particular, we identify $\gal(\Fhat/K(x_0)) = A\wr_{G_0} G$.
\item
\label{lem:wre_prod_c}
$\Fhat$ is regular over $L$.
\item
\label{lem:wre_prod_d}
Let $\Fhat_0 = \Fhat M$, $F_0 = FM$, and $N_0 = LM$. Then $N_0/L$, $\Fhat_0/M(x)$, and $F_0/M(x)$ are Galois, we have
\begin{align*}
\gal(N_0/M) &\cong G_0,\\
\gal(\Fhat_0/ M(x)) &\cong \Ind_{G_0}^G (A) \rtimes G_0,\\
\gal(F_0/M(x))& \cong A\rtimes G_0,
\end{align*}
and under these identification the maps in \eqref{eq:EPinHSI} are the restriction maps.
\item
\label{lem:wre_prod_e}
Let $\theta\colon \gal(K) \to A\wr_{G_0} G$ be a geometric weak solution, i.e., $\theta = \phihat^*$ for some place $\phihat$ of $\Fhat$ that is unramified over $K(x)$, and extends a specialization $x\mapsto a\in K$. Extend $\phihat$ to an $M$-place $\psihat$ of $\Fhat_0$. Then $\theta_0 := \rho\pi\theta|_{\gal(M)}  = (\psihat|_{F_1})^*$.
\end{enumerate}
\end{lemma}

\begin{proof}
Since the restriction map $\gal(F_1/M(x)) \to \gal(N_1/M)$ splits, we can lift $G_1 = \gal(N_1/M)$ to a subgroup of $\gal(F_1/M(x))$. Let $E$ be its fixed field. Then $E\cap N_1 = M$ and $EN_1 = F_1$. Let $y\in E$ be a primitive element of $E/M(x)$, i.e., $E = M(x,y)$. Then $y$ is also a primitive element of $F_1/N_1(x)$. Then the irreducible polynomial $f(x,Y)\in M[x,Y]$ of $y$ over $M(x)$ is Galois over $N_1(x)$.

Let $K'/K$ be a finite extension such that $N_1 \subseteq MK'$, the coefficients of $f(x,Y)$  are in $K'\cap M$, and $f(x,Y)$ is Galois over $K'(x)$. Let $L$ be a finite Galois extension of $K$ that contains $K'$.
Then, in particular, \eqref{lem:wre_prod_a} holds, $f(x,Y)\in L_0[x,Y]$ for $L_0 = L\cap M$, and $f(x,Y)$ is Galois over $L(x)$.
Choose a basis $c_1, \ldots, c_n$ of $L_0/K$ and let $x_0, x_1, \ldots, x_n$ be an $(n+1)$-tuple of variables.

Now by \cite[Lemma~3.1]{Haran1999Invent} there exist fields $F',
\Fhat'$ such that (denoting $\bfx = (x_1,\ldots, x_n)$)
\[
K(\bfx) \subseteq L_0(\bfx) \subseteq L(\bfx) \subseteq F'\subseteq \Fhat'
\]
realizes $A\wr_{G_0} G$ and $\Fhat'$ is regular over $L$. Furthermore, $F'$ is generated by a root of $f(\sum_{i=1}^n c_i x_i, Y)$ over $L(\bfx)$.

A usual combination of the Bertini-Neother Lemma and the Matsusaka-Zariski Theorem (see, e.g., \cite[Lemma~2.4]{Bary-Soroker2008IMRN}) implies that we can substitute $x_i$ with $\alpha_i+\beta_i x_0$ for some $\alpha_i,\beta_i\in K^\times$ and extend this substitution to a $K$-place of $\Fhat'$ with residue field $\Fhat$ with the same properties. That is, if we denote by $F$ the residue field of $F'$, then
\[
K(x_0) \subseteq L_0(x_0) \subseteq L(x_0) \subseteq F \subseteq \Fhat
\]
realizes $A\wr_{G_0} G$, $\Fhat$ is regular over $L$, and $F$ is generated by a root of $f(\sum_{i=1}^n c_i (\alpha_i + \beta_i x_0),Y)\in L_0[x_0,Y]$ over $L(x_0)$. Since $c_1,\ldots,c_n$ are linearly independent over $K$, we have that $\sum_{i=1}^n \beta_ic_i\neq 0$. Therefore $u=\sum_{i=1}^n c_i (\alpha_i + \beta_i x_0)$ is a transcendental element and $L_0(x_0) = L_0(u)$. We may apply an isomorphism that sends $u$  to $x$ to assume that 
\[
K(x_0) \subseteq L_0(x) \subseteq L(x) \subseteq F \subseteq \Fhat
\]
realizes $A\wr_{G_0} G$, $\Fhat$ regular over $L$, and $F$ is generated by the root $y$ of $f(x,Y)$. This finishes the proof of \eqref{lem:wre_prod_b} and \eqref{lem:wre_prod_c}.

Let $N_0=ML$, $F_0 = F M $, $\Fhat_0 = \Fhat L$, as in \eqref{lem:wre_prod_d}.
Note that $F_0  = FM =K(x,y) LM= N_0(x,y) $.
Since $\Fhat$ is regular over $L$, we have $L_0 = \Fhat\cap M = L\cap M$, and thus $\gal(N_0/M) \cong \gal(L/L_0)=G_0$. Similarly we have
\[
\gal(\Fhat_0/M(x)) \cong \gal(\Fhat/L_0(x)) \cong  \Ind_{G_0}^G(A)\rtimes G_0
\]
and $\gal(F_0/M(x)) \cong \gal(F/L_0(x)) \cong A\rtimes G_0 $. Since all isomorphisms are defined canonically, the maps in \eqref{eq:EPinHSI} are the restriction maps.

It remains to prove \eqref{lem:wre_prod_e}. Let $\phihat^*\colon \gal(K) \to \gal(\Fhat/K(x))$ be a weak geometric solution. Extend $\phihat$, $M$-linearly, to a place $\psihat$ of $\Fhat_0$. Then $\psihat$ is unramified over $M(x)$, and $\phihat^*|_{\gal(M)} = \psihat^*$. Moreover, since all maps in \eqref{eq:EPinHSI} are the restriction maps, we get that $\rho\pi (\phihat^*|_{\gal(M)}) = \rho\pi (\psihat^*) = (\psihat|_{F_1})^*$.
\end{proof}

\subsection{Separable extensions of fully Hilbertian fields}
\label{sec:sep-ext}
The following result is more general than Theorem~\ref{IT:separable}, cf.\ Main Theorem of \cite{Bary-SorokerHaranHarbater}.
\begin{theorem}
\label{thm:sep-ext}
Let $M$ be a separable extension of a fully Hilbertian field $K$. Then each of the following conditions suffices for $M$ to be fully Hilbertian.
\begin{enumerate}
\item
\label{case:sep-ext-finite}
$[M:K]<\infty$.
\item
\label{case:sep-ext-fin-gen}
The Galois group $\Gamma$ of the Galois closure of $M/K$ is finitely generated.
\item
The cardinality of the set of all finite subextensions of $M/K$ is less than $|M|$.
\item
$M$ is a proper finite extension of a Galois extension $N/K$.
\item
$M$ is an abelian extension of $K$.
\item
\label{case:sep-ext-diamond}
There exist Galois extensions $M_{1}, M_{2}$ of $K$ such that $M\subseteq M_{1}M_{2}$, but $M\not\subseteq M_{i}$ for $i=1,2$.
\item $M$ is contained in a Galois extension $N/K$ with pronilpotent group and $[M:K]$ is divisible by at least two primes.
\item $M/K$ is sparse.
\item $[M:K] = \prod p^{\alpha(p)}$, where $\alpha(p)<\infty$ for all $p$.
\end{enumerate}
\end{theorem}

\begin{proof}
Let $m=|K|$ and $(\mu_1 \colon \gal(M) \to \gal(N_1/M), \alpha \colon \gal(F_1/M(x)) \to \gal(N_1/M))$ be a rational split embedding problem for $M$. Then if we write $G_1 = \gal(N_1/M)$, then $\gal(F_1/M(x)) \cong A\rtimes G_1$ and $\alpha$ the quotient map. It suffices to find a set of $m$ pairwise-independent geometric solutions of $(\mu_1,\beta_1)$.

Let $K'/K$ be the finite separable extension given in Lemma~\ref{lem:wre_prod}. Choose a  `large' finite Galois extension $L/K$ such that $K'\subseteq L$ and let $\mu\colon \gal(K) \to G:=\gal(L/K)$ be the restriction map.
In the notation of  Lemma~\ref{lem:wre_prod} the embedding problem
\[
(\mu\colon \gal(K) \to G, \alpha\colon A\wr_{G_0} G \to G)
\]
is rational, and hence has a set of $m$ independent geometric solutions $\{\theta_i\mid i\in I\}$ of $(\mu,\alpha)$.

Cases~\eqref{case:sep-ext-finite} and \eqref{case:sep-ext-fin-gen}:
Note that \eqref{case:sep-ext-finite} is a special case of \eqref{case:sep-ext-fin-gen}. In \cite[Section~5.1.2]{Bary-SorokerHaranHarbater} it is proved that there exists an open normal subgroup $\Lambda$ of $\gal(K)$ (in the notation of loc.\ cit.\ $L=\Lambda$) such that if $\gal(L) \leq \Lambda$, then $S=\{\rho\pi \theta_i \mid i\in I\}$ is  a set of pairwise-independent solutions of $(\mu_1,\beta_1)$. Choose $L$ sufficiently large, so that it contains both $K'$ and the fixed field of $\Lambda$. Then on the one hand $S$ is a set of pairwise-independent solutions and on the other hand the solutions in $S$ are geometric (Lemma~\ref{lem:wre_prod}).

Case~\eqref{case:sep-ext-diamond}:
As in the previous case, in \cite[Section~5.1.3]{Bary-SorokerHaranHarbater} it is proved, basing on the proof of \cite[Theorem~25.4.3]{FriedJarden2005}, that there exists a normal subgroup $\Lambda$ of $\gal(K)$ such that if $\gal(L)\leq \Lambda$, then $S$ (as above) is a set of pairwise-independent solutions of $(\mu_1,\beta_1)$. Hence the same argument as above finishes the proof.

The proof of all the other cases other follows from the above cases, as shown in \cite{Bary-SorokerHaranHarbater} or proved in a similar manner.
\end{proof}

\begin{remark}
In the proof above, the only new property we exploited is that the
constructions of fiber products and twisted wreath product, and the
induced solutions, are field theoretic.
\end{remark}

\section{Purely inseparable extensions}
\label{sec:p.i.}
In the previous section we dealt with separable extensions of a fully Hilbertian field, using the twinning principle. Here we show that full Hilbertianity, as Hilbertianity, is preserved under purely inseparable extensions. This will be useful later, for example in order to prove that full Hilbertianity is preserved under all finite extensions.

\begin{proposition}
\label{prop:p.i-ful-hil}
Let $K$ be a fully Hilbertian field and $E/K$ a purely inseparable
extension. Then $E$ is fully Hilbertian.
\end{proposition}

\begin{proof}
Let $(\nu\colon \gal(E) \to \gal(M/E), \alpha \colon \gal(P/E(x)) \to \gal(M/E))$ be a rational embedding problem for $E$.
Since $E/K$ is purely inseparable and $F/E(x)$, $M/E$ are Galois, the extensions $F/K(x)$, $M/K$ are normal. Let $F$ (resp.\ $L$) be the maximal separable extension of $K(x)$ (resp.\ $K$) in $P$ (resp.\ $M$). Then $P=FE$ and $M=LE$. 
\[
\xymatrix{
E(x)\ar@{-}[r]
	&M(x)\ar@{-}[r]
		&P\\
K(x)\ar@{-}[r]\ar@{-}[u]
	&L(x)\ar@{-}[r]\ar@{-}[u]
		&F\ar@{-}[u].
}
\]
Then $P/F$ and $M/L$ are purely inseparable and $\gal(F/K(x)) \cong \gal(P/E(x))$, $\gal(L/K) \cong \gal(M/E)$ via the restriction maps.

Note that $P=FE$ is regular over $E$, i.e., linearly disjoint of the algebraic closure $\Egal$ of $E$ over $E$. As $F/K$ is separable and $E/K$ purely inseparable, $F$ and $E$ are linearly disjoint over $K$. Thus $F$ and $\Egal ( = \Kgal)$ are linearly disjoint. Thus $F$ is regular over $K$.

Let $\mu\colon \gal(K) \to \gal(L/K)$ and $\beta\colon \gal(F/K(x)) \to \gal(L/K)$ be the restriction maps.
As $K$ is fully Hilbertian, there exists a set $\{\phi^*_i \mid i\in I\}$ of independent solutions of $(\mu,\beta)$ with $|I|=|K|$. Extend $E$-linearly each $\phi_i$ to a place $\psi_i$ of $P$. Since $\mu = \nu \res_{E_s,K_s}$ and $\beta = \alpha \res_{P,F}$, and since the restriction maps $\res_{E_s,K_s} \colon \gal(E) \to \gal(K)$, $\res_{P,F}\colon \gal(P/E(x)) \to \gal(F/K(x))$ are isomorphisms, we get that $\{\psi^*_i\mid i\in I\}$ is a set of independent solutions of $(\nu,\alpha)$ of cardinality $|I|=|K|=|E|$.
\end{proof}

\begin{corollary}
\label{cor:finite-ext}
A finite extension of a fully Hilbertian field is fully Hilbertian.
\end{corollary}

\begin{proof}
Let $L/K$ be a finite extension. Let $E$ the maximal purely inseparable extension of $K$ inside $L$. Then $L/E$ is separable. By Proposition~\ref{prop:p.i-ful-hil} $E$ is fully Hilbertian and by Theorem~\ref{thm:sep-ext} $L$ is fully Hilbertian.
\end{proof}

\section{Dominating Embedding Problems}
This section provides some auxiliary results that reduce the necessary embedding problems needed for full Hilbertianity. This reduction is needed in the sequel in order to prove Theorem~\ref{IT:com-krull-dom}.

\subsection{Hurwitz's Formula}

Let $F/E$ be a finite separable extension of function fields over a base field $L$. Denote
by $g_F,g_E$ the genera of $F,E$, respectively, and let $n=[F:E]$.
Hurwitz's formula asserts that
\begin{equation}\label{hurwitz}
2g_F - 2 = n (2g_E - 2) + \deg R,
\end{equation}
where $R$ is the ramification divisor. 
If the characteristic of $L$ is $p>0$, then the primes of $F$ divide into two sets: the tamely ramified primes $P_{tr}=\{ \Pfrak \mid p\nmid e_\Pfrak\}$ and the wildly ramified primes $P_{wr} = \{ \Pfrak \mid p \mid e_\Pfrak\}$. If the characteristic of $L$ is $0$, then all the primes are, by definition, tamely ramified. We have
\begin{equation}\label{degR}
\deg R = \sum_{\Pfrak\in P_{tr}} (e_{\Pfrak}-1) + \sum_{\Pfrak\in P_{wr}} l_\Pfrak,
\end{equation}
and $l_\Pfrak > e_\Pfrak -1$, see \cite[Corollary~IV.2.4]{Hartshorne1977}.

If $\pfrak$ is a prime of $L(x)/L$, then either $\pfrak$ corresponds to some irreducible polynomial in $L[x]$, in this case we say that $\pfrak$ is finite (w.r.t.\ $x$), or $\pfrak$ corresponds to $\infty$, and then we write $\pfrak = \pfrak_\infty$.

The following result is an immediate consequence of Hurwitz's formula.

\begin{corollary}\label{cor:non-void_ram}
Let $F/L(x)$ be a separable extension of degree $n\geq 2$.  Assume that all primes of $F$ lying above $\pfrak_\infty$ are tamely ramified in $F$. Then there exists a prime $\Pfrak$ of $F$ lying over a finite prime $\pfrak$ of $L(x)$ such that $\Pfrak/\pfrak$ is ramified.
\end{corollary}

\begin{proof}

By \eqref{hurwitz} we have
\[
2g_F - 2 = -2n +\deg R \Rightarrow \deg R = 2g_F + 2(n-1) \geq 2(n-1) >n-1.
\]
(Recall that $n\geq 2$.)
Using the formula $\sum_{\Pfrak/\pfrak} e_\Pfrak f_\Pfrak = n$ with $\pfrak=\pfrak_\infty$ we get 
\[
\sum_{\Pfrak/\pfrak_{\infty}} (e_\Pfrak - 1) \leq n-1 < \deg R.
\] 
Thus there must be a prime $\Pfrak$ of $F$ not lying above $\pfrak_\infty$ that is ramified. 
\end{proof}

\begin{lemma}\label{lem:gen-by-ine}
Let $L$ be a field and let $F/L(x)$ be a finite Galois extension that is not wildly ramified at infinity.
Suppose $F/L(x)$ is regular, and that $\Ram(F/L(x)) =
\{c_1,c_2,\ldots,c_k\} \subseteq L$. For each $1 \leq i \leq k$
let $\frP_{i,1},\ldots,\frP_{i,g_i}$ be the primes of $F$ lying
over $(x-c_i)$. Denote the inertia group of $F/L(x)$ at $\frP_{i,j}$
by $I_{i,j}$. Then $G$ is generated by all the $I_{i,j}$.
\end{lemma}

\begin{proof}
Let $H\leq G$ be the subgroup generated by all $I_{j,j}$. Let $E = F^H$ the corresponding field. Then, by definition, $E$ is ramified at most at infinity, and this ramification, if exists, is tame. Thus by Corollary~\ref{cor:non-void_ram}, $E=L(t)$, i.e., $H=G$.
\end{proof}

Before we continue we recall some notation.
Fix a field $L$. Let $a\in \Lgal$, $f(X)$ its irreducible polynomial over $L$, and $\pfrak$ the prime of $L(x)$ that is defined by $f$. We say that $a$ is a \textbf{ramification point} of $F/L(x)$ if $\pfrak$ is ramified in $F$. We denote by $\Ram(F/L(x))\subseteq \Lgal$ the set of all ramification points of $F/L(x)$ in $\Lgal$. We emphasize that $\Ram(F/L(x))$ depends on $x$.

Let $\Pfrak/\pfrak$ be an extension of primes of $F/L(x)$ with ramification index $e$. If the characteristic of $L$ is positive, say $p>0$, we write $e=e' q$, where $\gcd(e',p)=1$ and $q$ is a power of $p$. We will call $e'$ the \textbf{tame ramification index} of $\Pfrak/\pfrak$ and $q$ the \textbf{wild ramification index}. If $L$ is of characteristic $0$, then the tame ramification index equals the ramification index, and the wild ramification index is $1$.

If $F/L(x)$ is Galois, then the tame and wild ramification indices are independent of the choice of the prime $\Pfrak$ that lies over $\pfrak$, so we abbreviate and say tame (resp., wild) ramification index of $\pfrak$ in $F$.

\begin{lemma}\label{lem:con_F}
Let $K$ be an infinite field, let $F_0/K(x_0)$ be a finite Galois extension, and let $L_0 =\Kgal\cap F_0$.  Assume that $\Ram (F_0/L_0(x_0)) \subseteq K_s$. Then there exist $x \in K(x_0)$ and a finite Galois extension
$F/K(x)$ such that, for $L = K_s\cap F$, the following conditions hold:
\begin{enumerate}
\item $K(x) = K(x_0)$.
\item The prime $\pfrak_\infty$ of $L(x)$ is tamely ramified in $F$.
\item $F_0\subseteq F$, \item $\Ram(F/L(x))\subseteq L$.
\item Let $c\in \Ram(F/L(x))$, $\Pfrak_c$ a prime of $F$ lying above $\pfrak_c = (x-c)$, and $e'$ (resp., $q$) be the tame (resp., wild) ramification index of $\Pfrak_c/\pfrak_c$.
Then
\renewcommand{\theenumii}{\theenumi \arabic{enumii}}
\begin{enumerate}
\item $(x-c)^{1/e'}\in F$,
\item the primitive $e'$-th root of unity $\zeta_{e'}$ is contained in $L$, and
\item the residue field of $F$ at $\Pfrak_c$ is a purely inseparable extension of $L$.
\end{enumerate}
\end{enumerate}
\end{lemma}

\begin{proof}
Since $K$ is infinite, there is a M\"obieus transformation that sends $x_0$ to $x=\frac{u + v x_0}{w + z x_0}$, for some $\left(\begin{matrix}u & v \\w & z\end{matrix}\right) \in PSL_2(K)$ under which $\infty$ is mapped to an unramified point. Clearly $K(x) = K(x_0)$ and $\Ram (F_0/L_0(x)) \subseteq K_s$.

Let $n = |\Ram(F_0/L_0(x))|$, write $\Ram(F_0/L_0(x)) =\{c_1,\ldots,
c_n\}$, and for each $i$,  let $e_i'$, (resp., $q_i$) be the tame (resp., wild) ramification index of $(x-c_i)$ in
$F_0$. Then $e_i = e_i' q$, $i=1,\ldots, n$ are the corresponding ramification indices.  Let $L$ be a finite Galois extension of $K$ that contains $L_0$, $c_1, \ldots, c_n$. 

Write $F_1=F_0
L$. Then $\Ram(F_1/L(x)) = \{c_1,\ldots, c_n\}$, and each $c_i$ has the same ramification indices  $e_1,\ldots, e_n$ (since
$L(x)/L_0(x)$ is unramified).
Let $F= F_1((x-c_i)^{1/e'_i}\mid i=1,\ldots, n)$. By
\cite[Example~2.3.8]{FriedJarden2005} it follows that $F/F_1$ is
unramified over the primes of $F_1$ that lie above $(x-c_i)$,
$i=1,\ldots, n$. By the multiplicity of ramification indices, we
get that each $(x-c_i)$ has ramification $e_i$ in $F$.

On the other hand, let $\pfrak \neq (x-c_i)$, $i=1,\ldots,n$ be a finite
prime of $L(x)$. Then  $\pfrak$ is unramified both in $F_1/L(x)$
and in $K((x-c_j)^{1/e'_i}\mid i=1,\ldots, n)$. Thus it is
unramified in their compositum $F$.
Replace $L$ with a larger extension, if necessary, to assume that the residue field of $F$ w.r.t.\ $\Pfrak_{c_i}$ is a purely inseparable extension of $L$.

Finally note that  the ramification of $\pfrak_\infty$ in $F$ divides $e_1' \cdots e_n'$, and hence is co-prime to the characteristic of $K$, if not $0$.
\end{proof}

\begin{proposition}
Let $K$ be an infinite field and
\[
(\mu_0 \colon \gal(K) \to \gal(L_0/K(x_0)), \alpha_0 \colon \gal(F_0/K(x_0) )\to \gal(L_0/K(x_0))
\]
 a non-trivial rational finite embedding problem. Then there
exist $x\in K(x_0)$ and a non-trivial rational finite embedding problem
\[
(\mu \colon \gal(K) \to \gal(L/K(x)), \alpha \colon \gal(F/K(x)) \to \gal(L/K(x))
\]
satisfying Properties (a)-(e) of Lemma~\ref{lem:con_F} such that an independent set
$\{\theta_i \mid i \in I\}$ of proper geometric solutions of
$(\mu,\alpha)$ induces an independent set $\{\res_{F,F_0}
\theta_i\mid i\in I\}$ of proper geometric solutions of
$(\mu_0,\alpha_0)$.
\end{proposition}

\begin{proof}
Let $F$ be the field given by Lemma~\ref{lem:con_F}. Consider the
commutative diagram
\[
\xymatrix{%
        &&{\gal(K)}\ar[d]^{\mu}\ar@/^30pt/[dd]^{\mu_0}\\
\gal(F/K(x))\ar[r]_-{\gamma}\ar@/^10pt/@{.>}[rr]^{\alpha}
    &\gal(F_0L/K(x))\ar[r]_-{\beta}\ar[d]^\pi
        &\gal(L/K)\ar[d]^{\mu'}\\
    &\gal(F_0/K(x))\ar[r]^-{\alpha_0}
        &\gal(L_0/K)
}%
\]
in which all the maps are the restriction maps. We are given an
independent set $\{\theta_i\mid i\in I\}$  of geometric proper
solutions of $(\mu, \alpha)$.

By Lemma~\ref{lem:basic-rat-eps_2}
$\{\gamma\theta_i\mid i\in I\}$ is an independent set of geometric
proper solutions of $(\mu,\beta)$. Then Lemma~\ref{lem:basic-rat-eps_2}
implies that $\{\pi\gamma\theta_i\mid i\in I\}$ is an independent
set of geometric proper solutions of $(\mu_0,\alpha_0)$.
\end{proof}

The following is an immediate conclusion.
\begin{corollary}\label{cor:suf}
For a field $K$ to be fully Hilbertian it
suffices that every non-trivial rational embedding problem
\[
(\mu \colon \gal(K)\to \gal(L/K), \alpha\colon \gal(F/K(x))\to \gal(L/K))
\]
satisfying Properties (a)-(e) of
Lemma~\ref{lem:con_F} has $|K|$ independent geometric solutions.
\end{corollary}

\subsection{Wild ramification}
This part deals with wild ramifications. The reader interested only with fields of characteristic zero can skip this part.

For the rest of this section fix a field $E$ with positive characteristic $p>0$ which is equipped with a discrete valuation $v$.

\begin{lemma}[Artin-Schreier theory]\label{Artin-Schreier}
Let $F/E$ be a Galois extension of degree $p$. There exists $a_0\in E$ such that for every $a\in a_0 + E^p - E$ and every $x$ satisfying $x^p-x=a$ we have
$F = E(x)$.
\end{lemma}

\begin{proof}
Let $\wp(x) = x^p - x$. Then the short exact sequence
\[
\xymatrix@1{
0\ar[r]& \bbF_p \ar[r] & E_s \ar[r]^-{\wp} & E_s\ar[r] &0
}
\]
of additive $\gal(E)$-modules gives rise to the following exact sequence of cohomology
\[
\xymatrix@1{
E \ar[r]^-{\wp} & E\ar[r] & H^{1} (\bbF_q) \ar[r] & H^{1} (E_s).
}
\]
By Hilbert's 90th, $H^1(E_s) = 0$. Since $\gal(E)$ acts trivially on $\bbF_p \cong \bbZ/p\bbZ$ we have $H^1(\bbF_q) = \Hom(\gal(E) , \bbZ/p \bbZ)$. Therefore $\Hom (\gal(E), \bbZ/p\bbZ) \cong E/\wp(E)$. Applying this to the restriction map $\gal(E) \to \gal(F/E)$ proves the assertion.
\end{proof}

We give a criterion for $v$ to be totally ramified in a Galois $p$ extension.

\begin{proposition}
\label{prop:crit-tot-ram} 
Let $F/E$ be a Galois extension of degree $p$. Then $v$ is totally ramified in $F$ if and only if there exists $a\in E$ and $x\in F$ such that $F=E(x)$, $x^p-x=a$, and $p\nmid v(a)<0$. 
\end{proposition}

\begin{proof}
First assume there exist $x\in F$ and $a\in E$ as in the proposition. Let $w$ be a normalized discrete valuation of $F$ lying above $v$. Then $w(x) < 0$, since otherwise $v(a) \geq 0$. We thus get that $w(x^p) < w(x)$, hence
\[
e(w/v) v(a) = w(a) =  w(x^p -x) = w(x^p) = p w(x).
\]
Since $\gcd(v(a),p)=1$, we get that $p$ divides $e(w/v)$, hence $p=e(w/v)$. Thus $v$ is totally ramified in each $F$.

Conversely, assume that $v$ is totally ramified in $F$ and let $w$ be its normalized extension. Fix $1 \leq i \leq m$. We show how to choose suitable $a,x$.  Choose a uniformizer $\pi \in E$ for $v$, i.e., $v(\pi)=1$. Then $w(\pi)=p$. Let $K$ be the residue field of $E$. Since $w/v$ is totally ramified, $K$ is the residue field of $F$ as well. Let $\phi$ be the corresponding place of $F$.

Apply Lemma~\ref{Artin-Schreier} to get an element $a_0 \in E$ such that for every $a \in a_0 + E^p -
E$ and $x \in F$ such that $x^p - x = a$ we have $F = E(x)$. Denote $A = a_0 + E^p -
E$. By \cite[Example~2.3.9]{FriedJarden2005} $v(A)$ is a set of negative integer numbers.
Let $l\in \bbN$ be minimal such that $-lp\in v(A)$. Choose some $a\in A$ satisfying $v(a) = -lp$. It suffices to show that $-lp\neq \max v(A)$.

Let $x\in F$ be a root of $X^p-X=a$. Then $F = E(x)$. We have $w(x^p - x) = w(a) =pv(a) < 0$. Hence $w(x)<0$, so $w(x^p)<w(x)$ which implies $pw(x) = w(x^p) = w(x^p-x) = p v(a)$. Therefore $w(x) = v(a) = -lp$.

Write $a = \pi^{-lp} \alpha$ and $x = \pi^{-l} \xi$. Note that $v(\alpha) = v(a) + lp=0$ and $w(\xi) = w(x) + l w(\pi) = -lp+lp=0$. Thus $\alpha,\xi$ are $w$-invertible, or equivalently, $\bar \alpha, \bar\xi \neq 0,\infty$, where the notation $\bar \cdot$ indicates the image after applying $\phi$. We apply $\phi$ to the equation
\[
\xi^p -   x \pi^{lp} = x^p \pi^{lp} -  x \pi^{lp}  = a \pi^{lp} = \alpha
\]
to get the equality $\bar \xi^p = \bar \alpha$.
Since $\Egag = K$, there exists $b\in E$ with $\bgag = \bar \xi$. In particular $v(b) = 0$. Also, as $\bgag^p - \bar \alpha =0$, we have $v(b^p-\alpha)>0$.
Let $a' = a - \pi^{-lp} b^p + \pi^{-l} b \in A$. Then $d:=v(a-\pi^{-lp} b^p) = v(-\pi^{-lp}(\alpha-b^p)) =-lp + v(\alpha - b^p)  > -lp$. Thus $v(a') \geq \min\{d, -m\} > -lp = v(a)$. This implies $v(a)\neq \max v(A)$, as needed.
\end{proof}

Now we are  ready to state and prove the main result of this part which sets a necessary condition for a totally ramified Galois extension of order a power of $p$ to remain totally ramified when moving to the residue extension w.r.t.\ some other place.

\begin{corollary}
\label{cor:wildly-ram}
Let $N/E$ be a Galois extension of degree a power of $p$.
Let $v$ be a discrete valuation of $E$ that is totally ramified in $N$.
Let $\phi$ be a place of $N$, let $\bar N,\bar E$ be the respective residue fields of $N,E$ under $\phi$, and assume the characteristic of $\bar E$ is $p$ and $\bar N/\bar E$ is a Galois extension of degree $[\bar N:\bar E] = [N:E]$. Then there exists $a_1,\ldots,a_m \in E$ with $v(a_i)<0$ and $\gcd(v(a_i),p)=1$ for each $1 \leq i \leq m$, such that for every discrete valuation $\vgag$ of $\Egag$, if $\agag_i:=\phi(a_i)\neq \infty$,
$\vgag(\agag_i)< 0$ and $\gcd(\vgag(\agag_i),p)=1$ for each $1 \leq i \leq m$, then $\vgag$ is totally ramified in $\Ngag$.
\end{corollary}

\begin{proof}

Let $F_1, \ldots, F_m$ be all the minimal extension of $E$ inside $N$. Then $F_i/E$ is a Galois extension of degree $p$. 
Proposition~\ref{prop:crit-tot-ram} gives $a_1,\ldots,a_m\in E, x_1\in F_1,\ldots,x_m \in F_m$ such that $F_i = E(x_i)$, $x_i^p - x_i = a_i$, $v(a_i)<0$ and $\gcd(v(a_i),p)=1$, for each $1 \leq i \leq m$.

Let $\Fgag$ be the residue field of $F$ under $\phi$. By assumption $\Ngag/\Egag$ is a Galois extension of the same degree as $N/E$, hence $\gal(\Ngag/\Egag) \cong \gal(N/E) = G$. Thus $\gal(\Ngag/\Fgag) = \Phi$. Let $H$ be the inertia group of $\vgag$ (w.r.t.\ some extension of it to $\Ngag$) in $G$, and let $\Egag'$ be its fixed field. Then, by definition $\vgag$ is unramified in $\Egag'$.

If $\Egag'\neq \Egag$, then $\Egag'$ contains some $\Fgag_i$. 
Since $\agag_i$ is finite, so is $\xgag_i$, $\Fgag_i=\Egag_i(\xgag_i)$, and $\xgag_i^p-\xgag_i =\agag_i$.
Proposition~\ref{prop:crit-tot-ram} implies that $\vgag$ is totally ramified in $\Fgag_i$, hence ramified in $\Egag'$. This contradiction implies that $\Egag'=\Egag$, i.e., $\vgag$ is totally ramified in $\Ngag$. 
\end{proof}

\section{Complete domains}\label{sec:gett-many-dist}

We now prove Theorems~\ref{IT:com-krull-dom} and \ref{IT:com-krull-dom-hil} about quotient fields of complete local
domains. To get independent solutions of a rational embedding
problem we show that if $a$ in a Hilbert set approximates all the
`ramification points' of the embedding problem w.r.t.\ some
valuations of our base field $K$, then the geometric inertia
groups become inertia groups w.r.t.\ those valuations of $K$.

\begin{lemma} \label{lem:unram-prime-val}
Let $L$ be a field, let $E/L(t)$ be a finite separable extension, and
let $\beta_1,\ldots,\beta_l \in E$. Let $\Pfrak/(t-c)$ be a prime extension of $E/L(t)$ such that $e(\Pfrak/(t-c)) =f(\Pfrak/(t-c)) =1$. Then there exists a finite set $\Lambda \subseteq L$, such that for any discrete valuation $v$ of $L$, satisfying $v(\lambda) = 0$ for each $\lambda \in \Lambda$, there exists $M \in \bbR$ such that for every $\alp \in L$ satisfying $v(\alp - c) > M$ the following properties hold:
\begin{enumerate}
\item
\label{cond:unramified}
There exists a prime $\frP_\alp$ of $E$ lying over $(t - \alp)$ such that $v$ is unramified in the residue field $\bar{E}_\alpha$.
\item
\label{cond:rel-prime}
Let $v_s$ be an extension of $v$ to $L_s$. Let $m_i$ be the order of $\beta_i$ at $\Pfrak$ and let $\nu\in \bbZ$ be co-prime to $m_i$, for each $1 \leq i \leq l$. Suppose $m_i < 0$ for each $1 \leq i \leq l$. Then if $v(\alpha - c)$ is co-prime to $\nu$, then $v_s(\bar \beta_i) \in \bbZ$, $v_s(\bar \beta_i) < 0$ and $\gcd(v_s(\bar \beta_i), \nu) = 1$, for each $1 \leq i \leq l$. (Here $\betagag_i$ is the image of $\beta_i$ modulo $\Pfrak_\alpha$.)
\end{enumerate}
\end{lemma}

\begin{proof}
Let $\hat{E} =\Ehat_\frP$ be the completion of $E$ at $\frP$, and consider the
completion $L((t-c))$ of $L(t-c)$ w.r.t.\ $(t - c)$. Then
$\hat{E}/L((t-c))$ is unramified and of residue field degree $1$.
Hence $\hat{E} = L((t-c))$.

Let $\beta_0 \in E$ be a primitive element of $E/L(t)$. Then
$\beta_0, \beta_1, \ldots, \beta_l \in L((t-c))$, thus we have
\[
\beta_i = \sum_{n=m_i}^{\infty} \lambda_{i,n} (t-c)^n, \qquad m_i\in \bbZ, \lambda_{i,n}\in L,
\]
for each $0 \leq i \leq l$. Without loss of generality we can assume that $m_0 > 0$ (by multiplying $\beta_0$ with a sufficiently large power of $(t-c)$). Put $\Lambda = \{ \lambda_{1,m_1},\ldots,\lambda_{l,m_l}\}$

Suppose $v$ is a discrete valuation of $L$, satisfying $v(\lambda)= 0$ for each $\lambda \in \Lambda$, and $v_s$ an extension of $v$ to $L_s$. Let $\hat{L}$ be the completion of $L$ at $v$, and extend $v$  to the algebraic closure of $\hat{L}$ compatibly with $v_s$. Since $\beta_i \in L((t-c)) \subseteq \hat{L}(((t-c)))$ is algebraic over $L(t-c)
\subseteq \hat{L}(t-c)$ for each $0 \leq i \leq l$, we get by \cite[Theorem 2.14]{Artin1968}
that $\beta_i$ converges at some $b_i \in \hat{L}$. That is,
$v(\lambda_{i,n} b_i^n)  = v(\lambda_{i,n}) + n v(b_i)\to \infty$.
Choose $b\in L$ with sufficiently large value $M := v(b)$. That is, we assume that $M \geq  \max\{v(b_0),v(b_1),\ldots,v(b_l)\}$, and hence $\beta_0,\beta_1,\ldots,\beta_l$ converge at $b$, and for every $r>M, 1 \leq i \leq l$ and $n>m_i$ the following inequality holds.
\begin{equation}
\label{eq:r-suff-large}
v(\lambda_{i,m_i}) + m_i r < v(\lambda_{i,n}) + nr.
\end{equation}

Denote $x = \frac{t - c}{b}$. Let $R = \hat{L}\{x\}$ be the ring
of convergent power series in $x$ over $\hat{L}$. That is, all
power series in $\hat{L}[[x]]$ whose series of coefficients
converges to $0$ with respect to $v$. In particular, $\beta_i =
\sum_{n = 1}^\infty {\lambda_{i,n} b^n} x^n \in R$, for all $i=0,1,\ldots,l$.

Now, suppose $\alp \in L$ satisfies $v(\alp - c) > M$. Then the
element $\eps = {\alp - c \over b} \in \hat{L}$ satisfies $v(\eps)
> 0$. Since $\hat{L}$ is complete with respect to $v$, we have
a substitution homomorphism $\phi_{\alp} \colon R \to \hat{L}$,
given by $\phi_{\alp} (\sum \mu_n x^n) = \sum \mu_n \eps^n$.
Extend $\phi_{\alp}$ to a place of $Q = \Quot(R)$.

We also have $\phi_\alpha(\beta_i) = \sum_{n=m_i}^{\infty} \lambda_{i,n} (\alpha-c)^n$. Applying \eqref{eq:r-suff-large} to $r = v(\alpha-c)$ we get that $v(\lambda_{i,m_i} (\alpha-c)^{m_i})< v(\lambda_{i,n} (\alpha-c)^n)$, for every $n>m_i$. Therefore, if we further assume that $v(\alp - c)$ is co-prime to $\nu$, then, for all $i\geq 1$,

\[
v(\phi_\alpha(\beta_i)) = v(\sum_{n=m_i}^{\infty} \lambda_{i,n} (\alpha-c)^n) = v(\lambda_{i,m_i} (\alpha-c)^{m_i}) = m_i v(\alpha - c)
\]
is a negative integer (since $m_i$ is negative), co-prime to $\nu$. 
This concludes the proof of \eqref{cond:rel-prime}.

Let us continue and prove \eqref{cond:unramified}. Since $\beta_0 \in R$, the field $E$ is contained in $Q$. The
restriction of $\phi_{\alp}$ to $E$ fixes $L$ and maps $t$ to
$\alp$ (by the definition of $\phi_{\alp}$). Let $\Pfrak_\alpha$ be the corresponding prime of $E$. Then $\Pfrak_\alpha$ lies over the prime $(t-\alp)$ of $L(t)$.

By the Weierstrass division theorem for convergent power series,
each element of $R$ can be written in the form $g \cdot (x - \eps) + r$,
with $g \in R, r \in \hat{L}$. It follows that the kernel of
$\phi_{\alp}$ (as a homomorphism of $R$) is a principal ideal,
generated by $x - \eps$. Moreover, by the Weierstrass preparation
theorem for convergent power series, $R$ is a unique factorization
domain (whose primes are the primes of $\hat{L}[x]$). Hence each
element $h \in Q$ can be written in the form $(x - \eps)^n {f
\over g}$, with $f,g \in R$ not divisible by $x - \eps$, and $n
\in \bbZ$. Then $\phi_{\alp}( {f \over g}) = {\phi_{\alp}(f) \over
\phi_{\alp}(g)} \in \hat{L}^\times$. Thus, if $n \geq 0$, then
$\phi_{\alp}(h) \in \hat{L}$, and if $n < 0$, then $\phi_{\alp}(h)
= \infty$. So the residue field of $Q$ at $\phi_{\alp}$ is
$\hat{L}$. Thus $\bar{E}_\alpha$ is contained in $\hat{L}$, hence
$\bar{E}_\alpha/L$ is unramified at $v$.
\end{proof}

\begin{lemma}
\label{lem:exact-cond}
Let $K$ be a field, let $F/K(x)$ be a finite Galois extension with $L=F\cap K_s$,
let $c\in \Ram(F/K(x)) \cap L$, let $\Pfrak$ a prime of $F$ lying above $\pfrak=(x-c)$, let $I$ be the corresponding inertia group, let $e$ be the ramification index,  and let $E = F^I$ the inertia field. Assume conditions (e1)-(e2) of Lemma~\ref{lem:con_F} are satisfied and that the residue field of $F$ at $\Pfrak$ is $L$, in addition to (e3).
Then there exists a finite set $\Lambda \subseteq L$, such that if $v_s$ is a valuation of $K_s$ that restricts to a discrete normalized valuation $v$ of $K$, and satisfies $v(\lambda) = 0$ for each $\lambda \in \Lambda$, then 
there exists $M\in \bbR$ such that if $\alpha\in K$ satisfies
\begin{enumerate}
\item $v(\alpha-c)> M$,
\item
\label{cond:val-alp-c=rho}
 $v(\alpha-c)$ is co-prime to $e$,
\item $(x-\alpha)$ is totally inert in $F$,
\end{enumerate}
then $\gal(\Fgag/ \Egag)$ is an inertia group of $v$ in $\Fgag$, where $\Fgag, \Egag$ are the respective residue fields of $F, E$ at the unique extension $\Pfrak_\alpha$ of $(x-\alpha)$ to $F$.
\end{lemma}

\begin{proof}
For any separable extension $N$ of $K$ we write $v_{N}$ for the normalized valuation that lies above $v$ and below $v_s$. Similarly for every subfield $F'$ of $F$ containing $L(x)$, we let $\Pfrak_{F'}$ be the restriction of $\Pfrak$ to $F'$.

We recall that $e = e' q$, where $e'$ is the tame ramification index and $q$ is the wild ramification index.
If the characteristic of $K$ is $0$, then $e'=e$ and $q=1$, and if the characteristic is $p>0$, then $\gcd(e',p)=1$ and $q=p^r$, for some $r\geq 0$.

\vspace{5pt}
\noindent\textsc{Construction in the wild ramification case:}
The following construction is needed only in the the case where $q>1$. Since $I$ is an inertia group, it has a unique $p$-sylow subgroup $J$. Then $J\normal I$, $|J| = q$, $(I:J) = e'$.
Let $E'$ be the fixed field of $J$ in $F$. Then $[F:E']=q$, $[E':E]=e'$. Furthermore both $\Pfrak_{E'}/\Pfrak_{E}$ and $\Pfrak_{F}/\Pfrak_{E'}$ are totally ramified, while the former ramification is tame and the latter is wild.

Let $a_1,\ldots,a_l \in E'$ be the elements that Corollary~\ref{cor:wildly-ram} gives when applying it to $F/E'$ (instead of $N/E$ in the notation of the corollary) with the normalized valuation $v_{\Pfrak_{E'}}$ (instead of $v$) arising from the prime $\Pfrak_{E'}$ and with the place $\phi$ arising from the prime $\Pfrak_\alpha$. Let $m_i = v_{\Pfrak_{E'}} (a_i)$ for each $1 \leq i \leq l$. Then $m_i<0$ is co-prime to $p$, for each $1 \leq i \leq l$. Let $\Lambda$ be the finite set given by Lemma~\ref{lem:unram-prime-val} (if $q = 1$, $\Lambda$ may be chosen to be empty). Suppose $v$ satisfies the conditions of the proposition (we fix such $v$ now, even if $q = 1$ and the construction we are making in this part is redundant), and let $M $ be the bound given by Lemma~\ref{lem:unram-prime-val}. Note that $\agag_i := \phi(a_i)$ is finite for all $\alpha$ but a finite set, hence, without loss of generality, we can assume it is finite (by enlarging $M$).

Then if
\begin{equation}
\label{eq:v-Egag'-cond}
p\nmid v_{\Egag'}(\agag) < 0,
\end{equation}
then $v_{\Egag'}$ is totally ramified in $F$. We now prove that \eqref{eq:v-Egag'-cond} holds.

Let $E_0 = L((x-c)^{1/e'})$.  We have
\[
e' q =e = e(\Pfrak_{F}/\Pfrak_{E_0}) e(\Pfrak_{E_0}/ \pfrak) = e(\Pfrak_{F}/\Pfrak_{E_0}) e',
\]
hence $e(\Pfrak_{F}/\Pfrak_{E_0})=q$.
\[
\xymatrix{
E\ar@{-}[r]^{e=e'}
	 &E'\ar@{-}[r]\ar@{.}@/^10pt/[rr]^{e=q}
			&E'E_0\ar@{-}[r]
				&F
\\
L(x)\ar@{-}[u]^{e=1}\ar@{-}[rr]_{e=e'}
	&
		&E_0\ar@{-}[u]\ar@{.}[ur]_{e=q}
}
\]
Since $e(\Pfrak_{E' E_0}/\Pfrak_{E_0})$ divides $e(\Pfrak_{E'}/\Pfrak_{E})=e'$, we get that $\Pfrak_{E' E_0}/\Pfrak_{E_0}$ is unramified. Similarly $\Pfrak_{E' E_0}/\Pfrak_{E'}$ is unramified. Hence $m_i$ is also the order of $a_i$ at $\Pfrak_{E' E_0}$ for each $1 \leq i \leq l$. Since $v(\alpha - c)$ is co-prime to $e$, it is co-prime to $p$, hence so is $v((\alpha - c)^{1 \over e'})$. Lemma~\ref{lem:unram-prime-val}\eqref{cond:rel-prime}, applied to $\beta_i=a_i, t=(x-c)^{1/e'}, \nu = p$ and with $c$ there replaced by $0$ here, we get that $v_{\Egag' \Egag_0}(\agag_i)<0$ is co-prime to $p$ for each $1 \leq i \leq l$. Since $[\Egag' \Egag_0 : \Egag']$ divides $e'$ which is relatively prime to $p$, \eqref{eq:v-Egag'-cond} follows.
For conclusion, Corollary~\ref{cor:wildly-ram} implies that $v_{\Egag'}$ is totally ramified in $\Fgag$.

\vspace{5pt}
\noindent\textsc{Conclusion of the proof:}
It remains to show that $v_{\Egag}/ v$ is unramified and that $v_{\Egag'}/v_{\Egag}$ is totally ramified. The former follows from Lemma~\ref{lem:unram-prime-val}\eqref{cond:unramified} (applied to the prime extension $\Pfrak_{E}/\pfrak$ of $E/L(x-c)$).

To prove that $v_{\Egag'}/v_{\Egag}$ is totally ramified, let $\tgag = (\alpha-c)^{1/e'}\in \Fgag$. Since $v(\alpha-c)$ is co-prime to $e'$, it is also co-prime to $e'$. Since 
\[
\bbZ \ni v_{\Fgag}(\tgag) = \frac{1}{e'} v_{\Fgag}(\alpha - c) =  \frac{e(v_{\Fgag}/v)}{e'} v(\alpha-c).
\]
we get that $e' \mid e(v_{\Fgag}/v)$. Note that the condition that $(x-\alpha)$ is totally inert in $F$ means that the degrees of the residue field extensions are preserved.
Since $[\Fgag:\Egag']=q$ is relatively prime to $e'$ and since $v_{\Egag}/v$ is unramified, we get that $e' \mid e(v_{\Egag'}/v_{\Egag})$. On the other hand, $e(v_{\Egag'}/v_{\Egag}) \leq [\Egag':\Egag] = e'$, so $e' = e(v_{\Egag'}/v_{\Egag})$, i.e., $v_{\Egag'}/v_{\Egag}$ is totally ramified.
\end{proof}

\begin{corollary}\label{cor:vel-geo-sol}
Let $K$ be a Hilbertian field.
Consider a rational finite embedding problem
$(\mu\colon \Gal(K)\to \gal(L/K), \alpha\colon \gal(F/K(x)) \to \gal(L/K))$ that satisfies (a)-(e) of Lemma~\ref{lem:con_F}. For each $c_i\in \Ram(F/L(x))$ let
$\frP_{i,1},\ldots,\frP_{i,g_i}$ be the primes of $F$ lying above
$\pfrak_i = (x-c_i)$. Assume that the residue field of $F$ at $\Pfrak_c$ is $L$.

Then there exists a finite set $\Lambda \subseteq L$, such that if $\{v_{i,j} \st 1 \leq i \leq k, 1 \leq j \leq g_i\}$ is a family of non-equivalent discrete valuations of $K$, totally split
in $L/K$, and trivial on $\Lambda$, then there exists a geometric solution $\Pfrak^*$ of $(\mu,\alpha)$ with solution field $L'$, and extensions $v'_{i,j}$ of the $v_{i,j}$ to $L'$, such
that $\Gal(L'/L)$ is generated by the inertia groups of the
$v'_{i,j}$ at $L'/L$. \end{corollary}

\begin{proof}
Denote the inertia group of $F/L(x)$ at $\frP_{i,j}$ by
$I_{i,j}$. By Lemma~\ref{lem:gen-by-ine}, $\gal(F/L(x))$ is generated by the $I_{i,j}$.

Let $E_{i,j} = F^{I_{i,j}}$. Then $\pfrak_{i,j}=\Pfrak_{i,j}\cap E_{i,j}$ is unramified over $L(x)$ and of degree $1$. Let $M_{i,j}$ be the bound given by Lemma~\ref{lem:exact-cond} applied to $\pfrak_{i,j}$, $E_{i,j}$, and $v_{i,j}$. Let $n_{i,j} \in \bbZ$ be co-prime to $e_i$ and such that $n_{i,j} > M_{i,j}$.
Since $v_{i,j}$ totally splits in $L/K$, $K$ is $v_{i,j}$-dense
in $L$. Choose $c_{i,j} \in K$ such that $v_{i,j}(c_{i,j} - c_i) >
n_{i,j}$. 

Since $K$ is Hilbertian, we may choose $\alp \in K$ such that
$(x-\alpha)$ is totally inert in $F$, or equivalently,
the
$K$-specialization $x \mapsto \alp$ extends to a place $\frP$ of $F$ which
defines a geometric solution $\Pfrak^*$  of $(\mu,\alpha)$. Moreover, since
Hilbert sets are dense with respect to any finite family of
valuations \cite[Proposition 19.8]{Jarden1991}, we may assume that $v_{i,j}(\alp
- c_{i,j}) = n_{i,j}$. Hence, $v_{i,j}(\alp - c_i) = v_{i,j}((\alp
- c_{i,j}) + (c_{i,j} - c_i)) = n_{i,j}$.

By Lemma~\ref{lem:exact-cond}, we get that $\gal(L'/\Egag_{i,j})\cong I_{i,j}$ is the inertia group of $v'_{i,j}/v_{i,j}$ in $L'$. Thus, since the $I_{i,j}$ generate $\gal(F/L(x))$, we get that $\gal(L'/L)$ is generated by the $\gal(L'/\Egag_{i,j})$, as asserted.
\end{proof}

\begin{proposition}
\label{prop:fully-Hilbertian}
Let $K$ be a Hilbertian field,
equipped with a family ${\calF}$ of discrete valuations,
satisfying:

\begin{enumerate}\renewcommand{\theenumi}{\roman{enumi}}
\item
\label{cond:fin-many-non-zero}
For each $a \in K^\times$, $v(a) = 0$ for almost
all $v \in {\calF}$.
\item
\label{cond:many-tot-split}
For each finite separable extension $L/K$, there
exist $|K|$ valuations in ${\calF}$ that totally split in
$L/K$.
\item
\label{cond:trv-prime-field}
Each $v \in {\calF}$ is trivial on the prime field
of $K$.
\end{enumerate}
Then the maximal purely inseparable extension $K_{ins}$ of $K$ is fully Hilbertian.
\end{proposition}

\begin{proof}
Let $\kappa = |K|$. Then any algebraic extension of $K$ has cardinality $\kappa$.
Let
\[
(\mu'\colon \gal(K_{ins})\to \gal(L'/K_{ins}) , \alpha' \colon \gal(F'/K_{ins}(x)) \to \gal(L'/K_{ins}))
\]
be a finite rational embedding problem for $K$. We need to produce $\kappa$ pair-wise independent geometric solutions of $\calE_{ins}$.

Let $K_0$ be a finite purely inseparable extension of $K$. Since $K_{ins}/K_0$ is purely inseparable, as in the proof of Proposition~\ref{prop:p.i-ful-hil}, the maximal separable extension $F$ (resp., $L$) of $K_0(x)$ (resp., $K_0$) in $F'$ (resp., $L'$) is a Galois extension, with same Galois group as $\gal(F'/K(x))$ (resp., $\gal(L'/K)$). Moreover, that proof shows that a set of $\kappa$ pair-wise independent geometric solutions of
\[
(\mu \colon \gal(K_0) \to \colon \gal(L/K_0), \alpha\colon \gal(F/K_0(x)) \to \gal(L/K_0))
\]
gives a set of  $\kappa$ pair-wise independent geometric solutions of $(\mu',\alpha')$. Therefore it suffices to find $\kappa$ pair-wise independent geometric solutions of $(\mu,\alpha)$.

By Corollary \ref{cor:suf} we may assume conditions (a)-(e) of Lemma~\ref{lem:con_F} hold. Note that condition (e4) gives that the residue field of $\Pfrak_c$, $c\in \Ram(F/L(x))$ is a purely inseparable extension of $L$. Thus if we replace $K_0$ with a sufficiently large finite extension we may assume that this residue field is $L$.

Since $K_0/K$ is purely inseparable, each $v\in \calF$ can be extended uniquely to a discrete valuation $v_0$ of $K_0$. Let $\calF_0 = \{ v_0 \mid v\in \calF\}$.
Clearly \eqref{cond:fin-many-non-zero} and \eqref{cond:trv-prime-field} hold for $\calF_0$.
The last assumption \eqref{cond:many-tot-split} also holds, since if $N/K$ is separable of degree $n$ and $v\in \calF$ is totally split in $N$, i.e., there exist $w_1,\ldots, w_n$ lying above $v$, then the corresponding $v_0\in \calF_0$ is totally split in $NK_0$, since $[NK_0:K_0]=n$, and the extensions of $w_1,\ldots, w_n$ to $NK_0$ lie above $v_0$.
Thus, without loss of generality, we may assume that $K=K_0$.

Let ${\calF}_L$ be the family of valuations in ${\calF}$ which are
totally split in $L/K$. Then $|{\calF}_L| = \kappa$. Let
$\{\theta_\del \mid \del \in \Del\}$ be a maximal family of pairwise-independent
geometric solutions of $(\mu,\alpha)$ with corresponding solution fields $L_i$. We wish to show that $|\Delta| = \kappa$.

Suppose $|\Del| < \kappa$. Each $L_\del$ is defined by a
polynomial over $K$. Adjoining all the coefficients of all these
polynomials to the prime field of $K$, we get that the solutions
are defined over a smaller subfield. That is, there exists a field
$M \subset K$ with $|M| < \kappa$, such that each $L_\del$ is of the
form $M_\del K$, for some finite Galois extension $M_\del$ of $M$.

Let ${\calF}' = \{v \in {\calF}_L \st v(a) = 0 \hbox{ for all }
a \in M^\times\}$. Then ${\calF}_L\hefresh {\calF}' = \bigcup
_{a \in M^\times} \{v \in {\calF}_L \st v(a) \neq 0\}$ is of
cardinality $|M|$ at most. In particular ${\calF}'$ is infinite. Let $\Lambda$ be the finite set given in Corollary~\ref{cor:vel-geo-sol}. By discarding finitely many valuations, we can assume that each $v \in {\calF}'$ is trivial on $\Lambda$.

Choose distinct $v_{i,j} \in {\calF}', 1 \leq i \leq k, 1 \leq j
\leq g_i$. Then all the $v_{i,j}$ are trivial on $M$, hence on all
$M_i$. By \cite[Lemma~2.3.6]{FriedJarden2005}, all of the $v_{i,j}$ are
unramified in $L_\del/K$ for each $\del \in \Del$.

By Corollary~\ref{cor:vel-geo-sol} there exists a geometric solution $\Pfrak^*$ with residue field $L'$, and
extensions $v'_{i,j}$-s of the $v_{i,j}$-s to $L'$, such that
$\Gal(L'/L)$ is generated by the inertia groups of the $v'_{i,j}$
at $L'/L$. Denote these groups by $I_{i,j}$. It suffices to show that $L'$ is
linearly disjoint from of each $L_\del$ over $L$, $\del \in \Del$, since then $\{\theta_\del \mid \del\in \Del\}$ is not maximal.

Indeed, let $\del \in \Delta$, and let $1 \leq i \leq k, 1 \leq j \leq
g_i$. The restriction $\Gal(L'/L) \to \Gal(L'\cap L_\del/L)$ maps
$I_{i,j}$ into the inertia group of $v'_{i,j}$ at $L' \cap L_\del
/L$. Since $v_{i,j}$ is unramified at $L_\del/L$, this inertia
group is trivial. Hence each element of $I_{i,j}$ fixes $L' \cap
L_\del$. So $I_{i,j} \subseteq \Gal(L'/L' \cap L_\del)$.

Since the $I_{i,j}$ generate $G$, $\Gal(L'/L' \cap L_\del) = G$,
hence $L' \cap L_\del = L$.
Thus $L'$ is linearly disjoint from $L_\del$ over $L$, for each
$\del \in \Delta$, as claimed.
\end{proof}

\begin{theorem} Let $K$ be the quotient field of a complete local
domain $R$ of dimension exceeding 1. Then its maximal purely inseparable extension $K_{ins}$ is fully Hilbertian.
\end{theorem}

\begin{proof} By Cohen's structure theorem for
complete local domains, $R$ is a finite extension of a domain $A$
of the form $K_0[[X_1,\ldots,X_n]]$ for some field $K_0$, or of
the $A = \bbZ_p[[X_1,\ldots,X_n]]$ for some prime integer $p$.
Since $\dim R > 1$, we have $n > 1$ in the first case and $n \geq
1$ in the second case. In both cases, \cite[\S5]{Paran2010} proves
the existence of a family of valuations satisfying conditions
(a)-(c) of Proposition \ref{prop:fully-Hilbertian}. Thus the maximal purely inseparable extension of
$\Quot(A)$ is fully Hilbertian. Since $K_{ins}$ is a finite extension of it,
$K_{ins}$ is also fully Hilbertian, by Corollary
\ref{cor:finite-ext}.
\end{proof}

Finally, we get Theorem~\ref{IT:com-krull-dom}:

\begin{corollary}
Let $K$ be the quotient field of a complete local
domain $R$ of dimension exceeding 1. Then $\Gal(K)$ is semi-free of rank $|K|$.
\end{corollary}

\begin{proof} Since $K$ is the quotient field of a complete
domain, every finite split embedding problem over $K$ is regularly solvable, by \cite{Paran2008} or \cite{Pop2009}. Hence every finite split embedding problem over $K_{ins}$ is regularly solvable.
By the preceding theorem $K_{ins}$ is fully Hilbertian, hence Proposition~\ref{prop:Cond-semi-free}, $\gal(K_{ins})$ is semi-free of rank $|K|$. Finally $\gal(K)\cong \gal(K_{ins})$ via the restriction map, hence $\gal(K)$ is semi-free.
\end{proof}

% ----------------------------------------------------------------
\bibliographystyle{amsplain}

\providecommand{\bysame}{\leavevmode\hbox to3em{\hrulefill}\thinspace}
\providecommand{\MR}{\relax\ifhmode\unskip\space\fi MR }
\providecommand{\MRhref}[2]{%
  \href{http://www.ams.org/mathscinet-getitem?mr=#1}{#2}
}
\providecommand{\href}[2]{#2}

\end{document}